\documentclass[a4paper,11pt]{amsart}
\usepackage{amsmath,amssymb,amscd,amsfonts}
\usepackage{amssymb}
\usepackage{amsrefs, esint}
\usepackage[T1]{fontenc}
\usepackage{tikz-cd}
\numberwithin{equation}{section}
\newtheorem{theorem}{Theorem}[section]

\newtheorem{rem}[theorem]{Remark}
\newtheorem{defn}[theorem]{Definition}
\newtheorem{lemma}[theorem]{Lemma}

\newtheorem{prop}[theorem]{Proposition}
\newtheorem{conj}[theorem]{Conjecture}
\newtheorem{cor}[theorem]{Corollary}
\newtheorem{example}[theorem]{Example}

\def\Vol{\mathop{\rm Vol}\nolimits}

\def\Re{\mathop{\rm Re}\nolimits} 
\def\Im{\mathop{\rm Im}\nolimits}

\def\dbar{\bar\partial}
\def\ddbar{\partial\bar\partial}
\def\d{\partial}

\def\cP{{\mathcal P}}

\def\cD{{\mathcal D}}

\def\cF{{\mathcal F}}
\def\cP{{\mathcal P}}
\let\ol=\overline

\let\ep=\varepsilon
\let\vp=\varphi 

\def\bC{{\mathbb C}}
\def\bR{{\mathbb R}}

\def\bP{{\mathbb P}}
\def\a{{\alpha}}
\def\z{{\zeta}}

\def\b{\beta}

\def\k{{\kappa}}

\def\l{{\ell}}

\title[Zero mass conjecture]
{On the residual Monge-Amp\`{e}re mass of plurisubharmonic functions, III: 
uniformly directional Lipschitz}

\author{Weiyong He}
\address{Department of Mathematics, University of Oregon, Eugene, OR, USA, 97403.}
\email{whe@uoregon.edu}

\author{Long Li}
\address{ Institute of Mathematical Science at ShanghaiTech University, 393 Middle Huaxia Road, Pudong, Shanghai, China, 201210.}
\email{lilong1@shanghaitech.edu.cn}

\author{Xiaowei Xu}
\address{School of Mathematical Sciences, USTC, Hefei, Anhui, China, 230026; CAS, Wu Wen-Tsun Key Laboratory of Mathematics.}
\email{xwxu09@ustc.edu.cn}

\begin{document}
\maketitle 

\begin{abstract}
The purpose of this article is to study the (residual) Monge-Amp\`{e}re mass 
of a plurisubharmonic function with an isolated unbounded locus. 
A general decomposition formula is obtained
under the Sasakian structure of the unit sphere.  
In complex dimension two, 
we obtain an $L^1$-apriori estimate on the complex Monge-Amp\`ere operator. 
This induces an upper-bound estimate on the residual mass, provided with the uniform 
directional Lipschitz continuity. 
As an application, 
 the zero mass conjecture is confirmed, 
 if the function further separates the circular direction in its alternating part. 
\end{abstract}

\smallskip

\section{Introduction}
\smallskip 

The zero mass conjecture,
raised by Guedj and Rashkovskii, 
has been recently studied with the circular symmetry, see \cite{Li23}, \cite{HLX}.
It is in fact proved in this case that 
the residual Monge-Amp\`ere mass
is bounded from the above by the Lelong number and the \emph{maximal directional Lelong number}
at the singularity.
Then the conjecture follows as a direct consequence of this upper-bound estimate. 
There have been many beautiful works related to this topic. 
For more details, the readers are referred to 
 \cite{G10}, \cite{Ra01}, \cite{Ra06}, \cite{Ra13}, \cite{Wik05}, \cite{BFJ07}, \cite{Ra16} and \cite{DGZ15}.

It turns out that  
the understanding of the \emph{standard Sasakian structure} 
on the unit sphere $S^{2n+1}$ is a key in the proof of the upper-bound estimate, 
see \cite{BG08}, \cite{SH62}.
More precisely, 
we obtain a \emph{decomposition formula} under the circular symmetry. 
It describes the push-forward  
of the complex Monge-Amp\`ere mass from the K\"ahler cone $(\bC^{n+1})^*$
to its base manifold $\bC\bP^n$ under the Sasakian structure, see Theorem 4.4 and Corollary 5.3 in \cite{HLX}.

In this paper, 
we continue our study on the (residual) Monge-Amp\`ere mass 
of a more general class.
That is, the main target is the family $\cP(B_1)$ 
consisting of all plurisubharmonic functions on the unit ball 
that have an isolated unbounded locus at the origin,
see Definition \ref{def-pre-001}. 
Furthermore, 
we also consider the sub-family $\cP^2(B_1)$ consisting of 
all functions in $\cP(B_1)$ that are $C^2$-continuous outside the origin,
see Definition \ref{def-pre-002}. 

First of all, 
we note that the natural decomposition of the standard complex structure on $\bC^{n+1}$
under the K\"ahler cone structure still works. 
In fact, this gives a new description 
of the plurisubharmonicity through the angle of Sasakian geometry, see Proposition \ref{prop-ch-001}.

Based on the new positivity conditions induced from the plurisubharmoncity, 
we utilize Sasakain geometry to
study the integral of the complex Monge-Amp\`ere operator of 
a function in $\cP^2(B_1)$ on a ball. 
Fortunately, 
a similar but more complicated decomposition formula has been obtained, see Theorem \ref{thm-df-001}. 
As expected, 
it boils down to the push-forward formula when the function is circular symmetric. 
For consistence of this paper, all computations in Sasakian geometry will be put 
into the last two Sections, see Appendix A and B.

However, there are several new difficulties 
to utilize this formula to approach the upper-bound estimate. 
This suggests that 
we need to take a closer look at the basic properties of plurisubharmonic functions with isolated unbounded locus. 
In fact,
there are several advantages to consider these functions with zero Lelong numbers, cf. \cite{SW24}, \cite{CW21}. 
Along this line, 
we give two sufficient conditions on  
zero maximal directional Lelong numbers under the circular symmetry, 
see Lemma \ref{lem-cs-001} and \ref{lem-cs-002}. 
Furthermore, 
it inspires us to introduce the \emph{strong maximal directional Lelong number},
see Definition \ref{def-pre-004}.

Another observation is that 
we can decompose a function $u$ in $\cP(B_1)$ into two parts as
\begin{equation}
\label{i-001}
 u: = u_s + v, 
\end{equation}
where $u_s$ is the $S^1$-invariant part, and $v$ is the alternating part of $u$. 
In fact, the function $u_s$ is the $S^1$-symmetrization of $u$, 
and then it is an element of the family $\cP(B_1)$ with the circular symmetry. 
Therefore,
the Lelong number
and the maximal directional Lelong number 
of $u$ are equal to the these numbers of $u_s$ at the origin, respectively. 

Consider the radial function $r$ in a polar coordinate and take $t: = \log r$. 
Notice that the $S^1$-invariant part $u_s$ always has a bounded $t$-derivative
due to the finiteness of the maximal directional Lelong numbers. 
On the contrary, the alternating part $v$ can oscillate in a dramatic way in general. 
Therefore, 
we introduce a condition called  \emph{uniform directional Lipschitz continuity} to control this oscillation. 
For a function $u\in \cP(B_1)$, 
this condition  
means that its first derivative $u_t $ is uniformly bounded in a smaller ball with radius $R$ centered at the origin.
That is, we define a finite number as 
\begin{equation}
L_A(u): = || r u_r ||_{L^{\infty}(\overline{B}_R)},
\end{equation}
for all $A = -\log R$, and then it gives a new quantity after taking the limit:
\begin{equation}
\k_u(0): = \lim_{r \rightarrow 0} || r u_r ||_{L^{\infty}(\overline{B}_r)}.
\end{equation}

As a first attempt, 
we restrict our case to the lower dimension $n=1$ from now on. 
Thanks to the Pohozaev inequality \cite{LW98} in harmonic map flows, 
the decomposition formula in $\bC^2$
gives an $L^1$-apriori estimate on the complex Monge-Amp\`ere operator.

\begin{theorem}[Theorem \ref{thm-mt-000}]
\label{thm-int-000}
Let $B_1$ be the unit ball in $\bC^2$. 
For a function $u$ in $\cP^{2}(B_1)$ and an $R\in (0,1)$, 
we have the estimate: 
\begin{equation}
\label{int-000}
\frac{1}{\pi^2}\int_{B_R} (dd^c u)^2 \leq 4 L_A(u)\cdot \left( \fint_{S_R} u_t \right), 
\end{equation}
where $t=\log r $ and  $A= -\log R$. 
\end{theorem}

Here we emphasize that the complex Monge-Amp\`ere operator is fully non-linear. 
It is a non-trivial fact that the integral of a second order fully non-linear operator 
can be controlled by the first derivative in the normal direction. 
This exactly reveals the nature why such an estimate is difficult to obtain in higher dimensions, 
see Section \ref{sec-61}.

Furthermore, equation \eqref{int-000} in fact has a geometric meaning. 
A smooth strictly plurisubharmonic function $u$ on $\bC^2$ 
can be viewed as a potential of the K\"ahler metric, i.e. $\omega_u: = dd^c u$. 
Then the usual complex Monge-Amp\`ere equation can be written as 
$$ \omega^2_u = e^F dV,$$
for some volume element $dV$.
From this point of view, 
the $L^1$-estimate in Theorem \ref{thm-int-000}
can be further interpreted as a volume comparison 
between the Monge-Amp\`ere measure and the Euclidean volume of a ball with radius $R$, 
see Corollary \ref{cor-mt-000}. 

Next we take the ball shrinking to the vertex of the K\"ahler cone, 
and then the upper-bound estimate on the residual mass follows.  

\begin{theorem}[Theorem \ref{thm-mt-001}]
\label{thm-int-001}
Suppose a function $u\in\cP(B_1)$ in $\bC^2$ is uniformly directional Lipschitz continuous 
near the origin. 
Then we have the following estimate:
\begin{equation}
\label{int-0011}
\tau_u(0) \leq 4 \k_u(0) \cdot \nu_u(0).
\end{equation}
\end{theorem}

It is clear that $\k_u(0)$ is finite if $u$ satisfies this continuity condition. 
As a direct consequence, 
the zero mass conjecture follows in this case,
see Corollary \ref{cor-mt-001}. 
As expected, 
all functions in $\cP(B_1)$ with circular symmetry 
satisfy this condition. 
However, 
there are still plenty of examples that fail to be uniformly directional Lipschitz continuous, 
see Example \ref{exm-001} and Example \ref{exm-002}.

Another approach
is to utilize the Fourier expansion along the $S^1$-direction on the alternating part of a function $u$. 
That is to say, we can write the function $v$ in equation \eqref{i-001} as 
\begin{equation}
\label{i-002}
 v \sim \sum_{k=1}^{\infty} \left( \cos k\theta \cdot v_k  +   \sin k\theta \cdot w_k  \right), 
\end{equation}
where $\theta$ is the argument in the $S^1$-direction, 
and $v_k$ and $w_k$ are circular symmetric functions. 
In fact, the $S^1$-invariant part $u_s$ exactly corresponds to 
the constant term in the above expansion. 

The simplest case in the Fourier expansion is that 
it only consists of a single frequency, e.g. 
$v = \cos k\theta \cdot v_k$
for some integer $k$.
Inspired by this fact, 
we introduce a new concept. 
That is, a function $u$ in $\cP(B_1)$ 
\emph{separates the $S^1$-direction} in the alternating part,
if equation \eqref{i-001} can be re-written as 
\begin{equation}
\label{int-003}
u = u_s + f(\theta) w,
\end{equation}
where $w$ is an $S^1$-invariant function, see Definition \ref{defn-sf-001}. 

However, 
the expansion in equation \eqref{i-002} really depends on the local coordinates. 
For this reason, equation \eqref{int-003} only needs to be satisfied on 
certain local trivialization of the Hopf-fibration,
called \emph{the proper complex Hopf-coordinates}. 
Note that these trivializations do not need to be $S^1$-invariant. 

With the aid of the strong maximal directional Lelong numbers, 
we prove that this kind of functions in fact satisfies the above continuity condition.



\begin{theorem}[Theorem \ref{cor-sf-001}]
Suppose a function $u\in \cP(B_1)$ separates the $S^1$-direction in its alternating part. 
Then it is uniformly directional Lipschitz continuous near the origin, 
and we have the estimate
$$  \tau_u(0) \leq 4 \k_u(0)\cdot \nu_u(0).  $$
In particular, the zero mass conjecture follows in this case. 
\end{theorem}

The next step is to consider a function $u$ 
that has only finite many terms in the Fourier expansion in its alternating part, 
i.e. the function $v$ in equation \eqref{i-002} has finite frequencies. 
Then we conjecture that a similar upper-bound estimate on the residual mass should be found,  
and hence the zero mass conjecture holds for such kind of functions,
see Conjecture \ref{conj-001}.

In a very recent work \cite{DLLWZ}, 
the upper-bound estimate in Theorem \ref{thm-int-001}
is enhanced to all higher dimensions, 
and the coefficients on the R.H.S. of equation \eqref{int-0011}
can be fixed to be $1$ in any dimension. 
However, 
it is still an interesting question to ask 
if we can obtain any $L^1$-apriori estimate on the complex Monge-Amp\`ere operators, 
as in Theorem \ref{thm-int-000}, in all higher dimensions.

\bigskip 

\textbf{Acknowledgment: }
The authors are very grateful to Prof. Xiuxiong Chen for his continuous encouragement and support 
in mathematics. It is also a great pleasure to thank 
Song Sun, Gao Chen, Per \AA hag, Haozhao Li, Wei Sun and Mingchen Xia for lots of useful discussions. 

Finally, we would like to thank 
Prof. Berndtsson and Prof. Rashkovskii for their kind reading on the draft of this paper and valuable suggestions.

\bigskip

\section{The maximal directional Lelong numbers}
\smallskip

Let $D$ be a domain in $\bC^{n+1}$ that contains the origin, 
and $D^*: = D -\{0 \}$ the punctured domain.
A plurisubharmonic function $u$ on $D$ 
is an upper semi-continuous function such that its restriction to each 
complex line in $D$ is subharmonic. 
In the following, 
we are going to consider a sub-collection of the family of all plurisubharmonic functions on $D$.

\begin{defn}
\label{def-pre-001}
A plurisubharmonic function on $D$ is in the family $\cP(D)$, if it is locally finite on $D^*$.
\end{defn}

In order to perform calculus on $D^*$, 
we further introduce the following family with better regularities.

\begin{defn}
\label{def-pre-002}
A plurisubharmonic function on $D$ is in the family $\cP^{2}(D)$, if it is $C^2$-continuous on $D^*$.
\end{defn}

It is clear that we have the inclusion $\cP^{2}(D)\subset \cP(D)$. 
Suppose a function $u$ is in the family $\cP(D)$.
Thanks to the upper semi-continuity, 
we further assume the normalization condition 
$$\sup_D u = -1.$$

Moreover, 
the \emph{unbounded locus} of a plurisubharmonic function
is the set of points that it is never bounded near these points. 
In other words, 
the unbounded locus of 
a $u\in\cP(D)$ is either empty or the origin.

Thanks to the Demailly-Bedford-Talyor product,   
the complex Monge-Amp\`ere operator 
is therefore well defined as a positive Borel measure, 
see \cite{BT0}, \cite{BT}, \cite{Dem93}. 
Then we have on the domain $D$ the following operator: 
$$ \mbox{MA}(u): = (dd^c u)^{n+1}, $$  
where we have used the normalized operator 
\begin{equation}
 d^c: = \frac{i}{2} (\dbar - \d), \ \ \ \mbox{and} \ \ \ d: = \d + \dbar. 
\end{equation}

Let $B_R$ be the ball centered at the origin with radius $0<R<1$, 
and then the residual Monge-Amp\`ere mass of $u$ is defined as a measure
\begin{eqnarray}
\label{pre-000}
\tau_u(0): &=& \frac{1}{ \pi^{n+1}} \mbox{MA}(u) (\{ 0\})
\nonumber\\
&=&  \frac{1}{\pi^{n+1}}  \lim_{R\rightarrow 0} \int_{B_R} (dd^c u)^{n+1}.
\end{eqnarray}

\subsection{The circular symmetry}
From now on,
we will take the domain $D$ to be the unit ball $B_1\subset \bC^{n+1}$.
Then a function $u$ is said to be \emph{circular symmetric} or \emph{$S^1$-invariant}, 
if it satisfies
$$ u(z): = u(e^{i\theta}z), $$
for all $\theta\in \bR$ and $z\in B_1$.
If a function $u\in\cP(B_1)$ is circular symmetric, 
then we say that it is in the family $\cF(B_1)$. 
For more details on circular symmetric plurisubharmonic functions, see \cite{BB}, \cite{Li19}.


The Lelong number (at the origin) 
of a plurisubharmonic function $u$ 
is defined via its average or maximal on the sphere $S_r: = \d B_r$, 
i.e. we take   
$$ S_u(r): = \frac{1}{a_{2n+1}} \int_{|\xi|=1} u(r\xi) d\sigma(\xi), \ \ \ $$
where $a_{2n+1}$ is the volume of the unit sphere, and 
$$ V_u(r): = \sup_{|\xi|=1} u(r\xi).  $$
It is a well known fact \cite{GZ} that $S_u$ and $V_u$ are both non-decreasing and convex function of 
the variable $t:= \log r \in (-\infty, 0)$, 
and then the Lelong number is
\begin{equation}
\label{pre-0001}
\nu_u (0): = \lim_{r\rightarrow 0^+} r\d_r^- S_u (r) =  \lim_{r\rightarrow 0^+} r\d_r^- V_u (r),
\end{equation}
and there is no difference to use left or right derivatives in the above equation. 

Furthermore, 
we can discuss the so called \emph{directional Lelong numbers} as follows. 
First, we note that all the complex directions in $\bC^{n+1}$ can be parametrized by points 
in the complex projective space $\bC\bP^n$, via the famous Hopf-fibration as 
\\

\begin{tikzcd}
\ \ \ \ \ \ \ \ \ \ \ \ \ \ \ \ \ \ \ \ \ \ \ \ \ \ \ \ \ S^1 \arrow[r, hook] & S^{2n+1} \arrow[r, "p"] & \bC\bP^n. 
\end{tikzcd}
\\

Fixing a point $\z\in \bC\bP^n$,
a complex line $\l_{\z}$ through the origin can actually be written as 
$$\l_{\z}: = \{\lambda\in \bC \ |  \ \  \lambda\cdot [\z] \ \}, $$
where $[\z]$ is its homogeneous coordinate in $(\bC^{n+1})^*$. 
Then the directional Lelong number (in the complex direction $\z$) is 
the Lelong number $\nu_{u|_{\l_\z}}(0)$ 
at the origin of the restriction $u|_{\l_\z}$, 
and we define it to be infinity, if the restriction is identically $-\infty$. 

It is also convenient to introduce a parametrization $(r, \theta, \z, \bar\z)$
of the space $(\bR^{2n+2})^*\cong (\bC^{n+1})^*$ induced by the fiber map $p$ of the Hopf-fibration. 
Then a function $u\in\cP(B_1)$ can be written as 
$$ u(z): = u(r, \theta, \z, \bar\z). $$
Here $r: = |z|$ is the radius function, and $\theta$ denotes the angle in the $S^1$-direction. 
More precisely, 
the three variables $(\theta, \z, \bar\z)$ together gives a local embedding of the unit sphere $S^3$,  
and a local trivialization like this is called a \emph{complex Hopf-coordinate}, see \cite{Li23}, \cite{HLX} for more details.  


Take a change of variables as $t:=\log r$. 
The maximal directional Lelong number $M_A(u)$ 
of a function $u\in \cF(B_1)$ at a distance $A>0$ is defined as 
\begin{equation}
 M_A(u): = \sup_{\z\in \bC\bP^n} \d^+_t u (\z)|_{t = -A}.   
 \end{equation}
Thanks to the log-convexity and non-decreasing properties of $u|_{\l_{\z}}$,
the number $M_A$ is certainly non-increasing in $A$. 
Moreover, it is proved in \cite{Li23} that 
it is always a non-negative real number. 
Then we can define the \emph{maximal directional Lelong number} $\lambda_u(0)$
of the function $u$ at the origin as the  limit: 
\begin{equation}
\lambda_u(0): = \lim_{A\rightarrow \infty} M_A(u), 
\end{equation}
and it must also be a non-negative real number. 
Moreover, the number $M_A(u)$ controls the variation w.r.t. the radius $r$ 
of the infimum of $u$ on the sphere $S_r$, see \cite{HLX}. 

A natural question is when 
this maximal directional Lelong 
number is equal to zero.  
For instance, we have the following two results. 

\begin{lemma}
\label{lem-cs-001}
For a function $u\in \cF(B_1)$, we have 
$$  u(0) > -\infty \Rightarrow \lambda_u(0) =0. $$
\end{lemma}
\begin{proof}
First we note that $u(z)\geq u(0)$ 
for all $z\in B_1$, since $u|_{\l_{\z}}$ is subharmoic and $S^1$-invariant
for each $\z\in \bC\bP^n$ fixed. 

Suppose not. 
Then there exists a real number $\delta >0$,
such that we have $M_A (u) > 2\delta$ for all $A>0$. 
Take a sequence $A_i \rightarrow +\infty$, 
and points $\z_i \in \bC\bP^n$ with the following condition: 
$$ \d_t^+ u(\z_i) |_{t = -A_i} > \delta.  $$
Consider the straight line 
$$ y_i(x): = \delta (x + A_i ) + u(0)  $$
It is clear that the convex curve $u|_{\l_{\z_i}}$ in $t$ is always above the line $y_i$, 
but then we have $y_i (-1) >0 $ for each $A_i$ large enough. 
This contradicts our normalization condition $\sup_{B_1} u = -1$.

\end{proof}

If a plurisubharmonic function is finite at a point, 
then its Lelong number is zero at the same point.  
Then we can further ask what happens to its maximal directional Lelong number, 
and a partial result is obtained.

\begin{lemma}
\label{lem-cs-002}
Suppose $u\in \cF(B_1)$ is continuous on $B^*_1$. 
Then we have 
$$ \nu_u(0) =0 \Rightarrow \lambda_u(0) =0.  $$
\end{lemma}
\begin{proof}
If $u$ has the zero Lelong number at the origin, 
then it is a well known fact that each directional Lelong number $\nu_{u|_{\l_{\z}}} (0)$ 
is zero for all $\z\in \bC\bP^n$. 
Then we can consider the following slope function:
\begin{equation}
\label{cs-001}
\phi_t (\z):=           t^{-1} u(e^t, \z, \bar\z)
\end{equation} 
This is a positive continuous function. 
Moreover, 
it is also non-decreasing and convex in $t$. 
Hence, for each $\z$ fixed,  we can take its decreasing limit as   
\begin{equation}
\label{cs-002}
\lim_{t \rightarrow -\infty} \phi_t (\z) = \nu_{u|_{\l_{\z}}} (0) =0. 
\end{equation} 
Then Dini's lemma implies that the convergence of $\phi_t$
is uniform. However, we also have the following control due to the convexity: 
\begin{equation}
\label{cs-003}
\phi_t = \frac{ 0 - u(e^t, \cdot)}{ 0 -t } \geq  \frac{ u(1, \cdot) - u(e^t, \cdot) }{0 -t} \geq \d^+_t u. 
\end{equation} 
In other words, we have 
$$ \sup_{\z\in \bC\bP^n} \phi_{-A} (\z) \geq  M_A (u), $$
and then our result follows. 
\end{proof}

In fact, this phenomenon in Lemma \ref{lem-cs-002} appears in a wider classes. 
\begin{example}
\label{exm-003}
Consider the following function in $\bC^2$ as 
\begin{equation}
\label{ex-017}
u(z^0,z^1):= -\left( -\log |z^0|^2 \right)^{1/2}. 
\end{equation}
Here $u(z)$ is a plurisubharmonic function near the origin, 
and its Monge-Amp\`ere measure is well defined in Cegrell's sense. 
It is also well known that its Lelong number and residual mass are both zero at the origin. 
Moreover, we can compute as follows:
\begin{equation}
\label{ex-018}
r\d_r u(z) = \frac{1}{\left( -\log |z^0|^2 \right)^{1/2}.  }
\end{equation}
It follows that we have $M_A(u) = A^{-1/2}$ for each $A$ large, 
and hence its maximal directional Lelong number $\lambda_u(0)$ is also zero. 
\end{example}

\subsection{The general case}
Comparing with the circular symmetric case, 
we can ask the question if there is any reasonable way 
to define the maximal directional Lelong number for a general $u\in \cP(B_1)$. 
One way is to decompose this function into 
$$ u: = u_s + v, $$
where $u_s$ is $S^1$-symmetrization of $u$, 
and it can be written as 
$$ u_s := \frac{1}{2\pi}\int_{S^1} u (e^{i\theta} z) d\theta.  $$
As the average of a plurisubharmonic function, 
the function $u_s$ itself is plurisubharmonic. 
In fact, it is clear that it is also in the family $\cP(B_1)$. 

Moreover, the function $v$ is the alternating part of $u$, i.e. 
we have 
$$ \int_{S^1} v (e^{i\theta} z) d\theta =0.  $$
Then we can define the maximal directional Lelong numbers as 
$$ M_A(u): = M_A(u_s); \ \ \ \mbox{and} \ \ \ \lambda_u(0): = \lambda_{u_s}(0).  $$

However, in order to understand the alternating part, 
we need more informations this time. 
Inspired from Lemma \ref{lem-cs-002},
we can consider the following function: 
\begin{equation}
\label{gc-000}
 \phi_t (\z): = \phi(t, \z)= \frac{\sup_{S^1} ( u|_{\l_{\z}} ) } {\log r},
 \end{equation}
for each $\z\in \bC\bP^n$ and all $t = \log r <0$. 
In other words, this is the slope of the supreme of the subharmonic function $u|_{\l_\z}$
over the $S^1$-direction, and then it can be rewritten under the complex Hopf-coordinate as 
\begin{equation}
\label{gc-001}
\phi_t(\z) =  t^{-1} \hat u (t, \z), 
\end{equation} 
where we take  
\begin{equation}
\label{gc-0011}
\hat u (t, \z): = \sup_{\theta} u(e^t, \theta, \z, \bar\z). 
\end{equation}

Then it is a standard fact for subharmonic functions 
that $\hat u$ is a non-decreasing and convex function of $t$, 
and the limit of $\phi_t$ is exactly the directional Lelong number, i.e. we have 
\begin{equation}
\label{gc-002}
\nu_{u|_{\l_\z}}(0) = \lim_{t\rightarrow -\infty} \phi_t(\z).
\end{equation} 
Therefore, it is reasonable to introduce the following concept.

\begin{defn}
\label{def-pre-004}
For a function $u\in \cP(B_1)$, 
the strong maximal directional Lelong number of $u$
at a distance $A>0$ is defined as 
$$ N_A(u): =  \sup_{\z\in \bC\bP^n} \phi_{-A} (\z).$$
\end{defn}
In priori, this number could be infinitely large. 
However, it is not difficult to see that it must be finite for all $A$ positive. 

\begin{lemma}
\label{lem-gc-001}
For a function $u\in \cP(B_1)$, we have for all $A$ positive
$$ 0\leq N_A(u)<  +\infty.  $$
\end{lemma}
\begin{proof}
The function $\phi_t$ is non-decreasing in $t$, 
and then the number $N_A$ is non-increasing in $A$. 
Hence
it is enough to prove the finiteness 
for a certain positive number $A_0$. 
This is because our function $u$ is locally bounded 
on the punctured ball $B^*_1$, 
and then there is a pointwise lower bound $C_{A_0}$ of $u$ on the sphere  
such that we have 
$$ u|_{S_{R_0}} \geq -C_{A_0}, $$
where $R_0:= e^{-A_0}$.
It follows that we further have 
\begin{equation}
\label{gc-003}
 \phi_{-A_0}(\z) \leq  \frac{ C_{A_0} }{ A_0},
\end{equation} 
for all $\z\in \bC\bP^n$. Then our result follows. 

\end{proof}

Thanks to the convexity of $\hat u$ in $t$ again, 
we can further obtain the following estimate as in equation \eqref{cs-003}: 
\begin{equation}
\label{gc-004}
N_A (u) \geq \sup_{\z\in \bC\bP^n}  \d^+_t \hat u|_{t= -A}.
\end{equation} 
This also implies that the R.H.S. of equation \eqref{gc-004}
is also finite for all $A$ positive. 
Finally, 
we further introduce the following concept
for a later purpose.

\begin{defn}
\label{def-pre-003}
A function $u\in \cP(B_1)$ is said to be uniformly directional Lipschitz continuous near the origin,
if the derivative $ ru_r$ is an $L^{\infty}$-function in a smaller ball, i.e. 
we have 
$$ L_A(u): = || ru_r ||_{L^{\infty}(\overline{B}_R)} < +\infty, $$
for some large constant $A: = -\log R$. 
\end{defn}

It is clear that the function $L_A$ is non-increasing in $A$,
and hence we can also define the directional Lipschitz constant 
of $u$ at the origin as
$$ \k_u(0): = \lim_{A\rightarrow +\infty}  L_A (u) \in [0 +\infty]. $$
In other words, the non-negative number $\k_u(0)$ is finite, if $u$ 
is uniformly directional Lipschitz continuous near the origin. 
When the function $u$ is furthermore $S^1$-invariant, 
we can see that this definition boils down to the former one:
\begin{equation}
\label{gc-0050}
 L_A(u) = \sup_{A\leq B} M_B(u) = M_A(u), 
 \end{equation}
since $M_A(u)$ is non-increasing in $A$.
Hence $\k_u(0)$ is exactly the maximal directional Lelong number $\lambda_u(0)$ in this case.
That is, a function $u$ in $\cF(B_1)$ is always uniformly directional Lipschitz continuous.

However, this number is not always finite for the general case. 
For instance, consider the following function in the unit disk in $\bC$: 
\begin{equation}
\label{gc-005}
f(z): = \sum_{k=1}^{\infty} \frac{1}{k^2} \log \left( | z - k^{-1} |^2 + e^{-k^5} \right). 
\end{equation}

This subharmonic function is well-known because it is finite everywhere,
but its unbounded locus is the origin, and hence is discontinuous there. 
Moreover, we note that 
its derivative $ru_r$ grows like 
$r^{-3}$ on the real axis towards the origin.  
Hence it is not uniformly directional Lipschitz continuous. 
In fact, 
we can also use it to construct more examples in higher dimensions. 

\begin{example}
\label{exm-001}
The following construction is inspired from Example 4.6, \cite{ACH15}. 
Let $z: = (z^0, z^1)$ be the complex rectangle coordinate on $\bC^2$, 
and then we define 
\begin{equation}
\label{ex-001}
u(z): = \max\left\{ f(z^0), \log |z^1|    \right\}.
\end{equation}

This function $u$  is plurisubharmonic in the unit ball. 
It is everywhere finite, and its unbounded locus is also the origin. 
Hence it is in the family $\cP(B_1)$. 
However,
its restriction to the complex line $\{ z^1=0 \}$ is exactly the function $f$.  
Therefore, it is not uniformly directional Lipschitz continuous near the origin.



\end{example}

Finally, we note that there is a technique \cite{ACH19}  to 
replace our function by a larger pluri-complex Green function 
with the same residual mass. 
In fact, given a general function $u\in \cP(B_1)$, 
we can take its envelope as 
 $$ u^r: = \sup \left\{  v\in \cP(B_1); \ \ v\leq u \  \mbox{on} \   B_r \right\}, $$
and then the sequence $\{ u^r \}$ increases pointwise to a function $\tilde{u}\in \cP(B_1)$. 
It is clear that we have $\tilde{u} \geq u$,
and hence $\nu_{\tilde{u}}(0)=0$ due to Demailly's comparison theorem. 
Moreover, it satisfies 
$$ (dd^c \tilde{u})^n = \tau_u(0) \delta_0. $$

Therefore, it is interesting to ask whether a pluri-complex Green function in $\cP(B_1)$
is uniformly directional Lipschitz continuous or not.

\bigskip

\section{The lower dimensional case}
\smallskip

In this section,
we will study the residual Monge-Amp\`ere mass 
of a function $u\in \cP^{2}(B_1)$ in complex dimension two, 
i.e. we assume $n=1$. 
The reason is that 
the decomposition formula in Theorem \ref{thm-df-001}
only contains two terms, and 
there is no higher order of the $(1,1)$-form $\Theta$
appearing. 
Therefore, it is supposed to be the most fundamental case.

\subsection{Basic properties}
Before moving on, 
we will invoke some useful properties 
of a plurisubharmoinc function with isolated unbounded locus. 
The following results, Lemma \ref{lem-mt-001} - Lemma \ref{lem-mt-005}, 
were first stated in \cite{Li23} with the assumption of circular symmetry. 
However, the proves actually do not account on any symmetry condition.  
For this reason, we only recall them as follows. 

\begin{lemma}
\label{lem-mt-001}
For a function $u\in \cP(B_1)$, the complex Monge-Amp\`ere 
mass of its standard regularization $u_{\ep}$ converges on a ball as 
$$ \emph{MA}(u)(B_R) = \lim_{\ep\rightarrow 0} \emph{MA}(u_{\ep}) (B_R), $$
for almost all $0< R <1$. 
\end{lemma}

The convergence in the above Lemma actually holds for every $R\in (0,1)$, 
provided with $u\in \cP^{2}(B_1)$. Hence we further have the following result via the Stokes Theorem.

\begin{lemma}
\label{lem-mt-002}
For a function $u\in \cP^{2}(B_1)$, we have for all $R\in (0,1)$ 
$$ \int_{B_R} (dd^c u)^{n+1} = \int_{S_R} d^c u \wedge (dd^c u)^n. $$
\end{lemma}

Then we  conclude the following Friedrichs's type estimate. 

\begin{lemma}
\label{lem-mt-003}
For a function $u\in \cP(B_1)$, and a point $z\in B^*_{1-2\delta}$, we have 
$$ | r\d_r (u * \rho_{\ep})(z) - r (\d_r u* \rho_{\ep})(z) | \leq 2\ep || \nabla u ||_{L^1(B_{1-\delta})} $$
\end{lemma}
It is known by Blocki \cite{Blo04} 
that a function $u\in \cP(B_1)$ is in fact in the Sobolev space $W^{1,2}_{loc}(B_1)$. 
Therefore, we can write 
\begin{equation}
\label{mt-001}
r (\d_r u)_{\ep}(z) = u_{\ep, t} + O(\ep),
\end{equation}
for all $z$ in a smaller ball in $B_1$. 
Hence we have the following convergence due to the slicing theory.

\begin{lemma}
\label{lem-mt-004}
For a function $u\in \cP(B_1)$, and any constant $1< A < B$, 
there exists a subsequence $u_{\ep_k}$ such that 
the following integral on the boundary sphere $S_r$ convergences  as $k\rightarrow +\infty$
$$ \int_{\bC\bP^1}\left( \int_{S^1} u_{\ep_k, t} \ d\theta \right) \omega_{FS} 
\rightarrow \int_{\bC\bP^1}\left( \int_{S^1} u_t  \ d\theta \right) \omega_{FS}, $$
for almost all $t: =\log r \in [-B, {-A}]$. 
\end{lemma}

Finally, we need a slightly stronger version of Lemma 6.6 in \cite{Li23}.

\begin{lemma}
\label{lem-mt-005}
Let $u$ be a function in $\cP(B_1)$. 
Given any constant $A>1$, 
there exists a uniform constant $C>0$, may depend on $u$, such that we have the estimate: 
$$ L_A (u_{\ep}) \leq  L_A(u) + C\ep, $$
for all $\ep < \ep_0$ where $\ep_0: = (e^{-1} - e^{-A})/2$.
\end{lemma}
\begin{proof}
First we note that the first order linear operator $r\d_r$ can be rewritten in real coordinates as 
$$ r\d_r u  = \sum_{m} x^{m} \d_{x^m} u,  $$
where $m=1,\cdots, 2n+2$.
For a function $u\in\cP(B_1)$ and a point $x\in \overline{B}_r$, we can first commute the convolution and the 
derivatives as  
\begin{equation}
x^m \d_{x^m} (u * \rho_{\ep}) =  x^m (\d_{x^m} u * \rho_{\ep}).  
\end{equation}
Then we further compare the following two terms as 
\begin{eqnarray}
\label{mt-003}
&&\left|   x^m ( \d_{x^m} u * \rho_{\ep}) -  (x^m \d_{x^m} u) * \rho_{\ep} \right|
\nonumber\\
&\leq& \int_{|y|<1} \left| x^m - (x^m - \ep y^m)  \right|  |\d_{x^m}u (x-\ep y)| \rho(y) d\lambda(y)
\nonumber\\
&\leq& \ep \int_{B_{\ep}(x)} |\nabla u | d\lambda.
\end{eqnarray}
Hence we obtain 
\begin{equation}
\left|  (r\d_r u)_{\ep} - r\d_r u_{\ep} \right| \leq \ep || \nabla u ||_{L^1(\overline{B}_{r+\ep})}.
\end{equation}
Therefore, our result follows from Blocki's estimate.


\end{proof}

\subsection{The upper-bound estimates}
In $\bC^{n+1}$, the positivity condition of a plurisubharmonic function 
can be re-characterized by the standard Sasakian structure of the hypersphere $S^{2n+1}$,
see Proposition \ref{prop-ch-001}. 
In particular, the complex Hessian of a $C^2$-continuous plurisubharmonic function $u$
is represented by the following semi-positive $(2\times 2)$-matrix in $\bC^2$: 
$$
S(u): = 
\begin{pmatrix}
\Delta (u|_{\l_{\z}})/4 & X_0 \bar X_{1} u \\
\\
X_{1} \bar X_0 u &  X_{1}\bar X_{1} u + r (\bar X_0 u) h_{\z\bar\z}
\end{pmatrix}, 
$$
where $\z: = z^1/ z^0$ is the holomorphic coordinate on $\bC\bP^1$. 
Here $h$ is the local basic function of the complex Hopf-coordinate  defined by 
$$ h(\z, \bar\z): = \log(1+ |\z|^2) - \log|\z|, $$
and the vector fields $\{ X_0, X_1, \bar X_0, \bar X_1 \}$
build a local frame of the tangent bundle of the K\"ahler cone $\bC^2- \{0\}$, 
see equations \eqref{bf-008}-\eqref{bf-009}. 

In fact, we can also write down everything under the complex Hopf-coordinate $(r, \theta, \z, \bar\z)$: 
\begin{equation}
\label{ld-000}
 \Delta (u|_{\l_{\z}}) = e^{-2t} (u_{tt} + 4 u_{\theta\theta}), 
\end{equation}
\begin{eqnarray}
\label{ld-001}
&& X_{1} \bar X_{1} u + r(\bar X_0 u) h_{\z\bar \z}
\nonumber\\
&=& u_{\z\bar\z} + \frac{1}{2} u_t h_{\z\bar\z} + i (h_\z u_{\theta\bar\z} - h_{\bar\z} u_{\theta\z}) + | h_\z |^2  u_{\theta\theta}. 
\end{eqnarray}
Moreover, the mixed term can be written as 
\begin{equation}
\label{ld-002}
 X_{1} \bar X_{0} u
= \frac{1}{2r} \left\{ u_{t\z} - 2h_{\z}  u_{\theta\theta}   + i ( 2 u_{\theta\z} + h_{\z} u_{t\theta})  \right\}.
\end{equation}
We note the two terms in equation \eqref{ld-000} and \eqref{ld-001}
are both non-negative due to the positivity condition. 

Next we are going utilize the decomposition formula in $\bC^2$
to study the upper-bound estimate. 
First, we compute as follows: 
\begin{equation}
\label{ld-003}
r \bar X_0 u + r X_0 u = u_t,
\end{equation}
and 
\begin{equation}
\label{ld-004}
( r \bar X_0 u )^2 + ( r X_0 u)^2 = \frac{1}{2} (u_t^2 - 4 u_{\theta}^2).
\end{equation}
Thanks to the decomposition formula in Theorem \ref{thm-df-001}, 
the complex Monge-Amp\`ere mass is reduced to 
\begin{eqnarray}
\label{ld-005}
&& 4 \int_{S_r} d^cu \wedge dd^c u
 \nonumber\\
 &=& 2 \int_{S_r} u_t \left\{ X_1 \bar X_1 u +  r(\bar X_0 u) h_{\z\bar\z} \right\} d\theta \wedge (i d\z \wedge d\bar\z)
 \nonumber\\
 &-& \int_{S_r} \left( |u_t|^2 - 4 | u_{\theta}|^2 \right)  d\theta \wedge \omega_{FS}. 
\end{eqnarray}

Then we are going to investigate the second term
on the R.H.S. of the above equation for each $\z\in \bC\bP^1$ fixed, 
i.e. we can write it as 
$$  \int_{0}^{4\pi} (|u_t|^2 - 4 |u_{\theta}|^2)  d\theta = \int_{S^1} ( |u_t|^2 - |\xi_0 u|^2 ). $$
Here we need to invoke a famous inequality in harmonic map flows, see Lemma 2.4, \cite{LW98}. 
A proof is provided for the convenience of the readers. 


\begin{lemma}[Pohozaev inequality]
\label{lem-sub-002}
Let $u$ be a $C^2$-continuous subharmonic function on the unit disk.
Then the following integral along the boundary of a smaller disk $D_R$ for some $0< R <1$
can estimated as 
\begin{equation}
\label{sub-007}
\left| \int_{S^1} \left( |u_t|^2 - |\xi_0 u|^2 \right)  \right| \leq 2 L_A(u) \cdot \left( \int_{S^1} u_t \right), 
\end{equation}
where $A = -\log R$. 
\end{lemma}
\begin{proof}
For simplicity, we will slightly change our notations in this proof,  
and hope that it will be clear enough.
Let $z: = r e^{i\theta}$ be a complex coordinate on $\bC$, 
and also take $t: = \log r$ for all $r>0$. 
Then we perform the following integration by parts as 
\begin{eqnarray}
\label{sub-004}
 - \int_{D_R} ( ru_r ) dd^c u 
&=&  \int_{D_R} u_r dr\wedge d^c u + \int_{D_R} r d u_r \wedge d^c u -  \int_{\d D_R} (ru_r) d^c u 
\nonumber\\
&=&  \int_{D_R} du \wedge d^c u + \int_{D_R}  ( r d u_r - u_{\theta} d\theta ) \wedge d^c u 
- \int_{S^1} |u_t|^2 d\theta, 
\end{eqnarray}
where $d^c : = i (\dbar - \d)$, 
and we remember that the integrals taken over $S^1$ are in fact on the boundary circle $\d D_R$. 
Then it follows the equation:
\begin{eqnarray}
\label{sub-005}
 ( r d u_r - u_{\theta} d\theta ) \wedge d^c u 
&=&  \left( r^2 u_{rr} u_r - r^{-1} u^2_{\theta} + u_{r\theta} u_{\theta} \right) dr\wedge d\theta
\nonumber\\
&=& \frac{1}{2}   r^2 \d_r (u_r^2 + r^{-2} u^2_{\theta})  dr \wedge d\theta.
\end{eqnarray}
Hence perform integration by parts again, and we obtain 
\begin{eqnarray}
\label{sub-006}
&&  \int_{D_R} ( r d u_r - u_{\theta} d\theta ) \wedge d^c u 
\nonumber\\
&=&  \frac{1}{2} \int_{S^1} \left( |u_{\theta}|^2 + |u_t|^2  \right) d\theta - \int_{D_R} du \wedge d^c u. 
\end{eqnarray}
Equipped with equation \eqref{sub-004} and \eqref{sub-006}, 
we have the equation 
\begin{equation}
\label{sub-0071}
2\int_{D_R} ( ru_r ) \Delta u
 = \int_{S^1} \left( |u_t|^2 - |u_{\theta}|^2 \right).
\end{equation}
Then our result follows from the definition of $L_A$, see Definition \ref{def-pre-003}. 
\end{proof}
Now we are ready to obtain 
a uniform estimate  on the 
upper bound of the complex Monge-Amp\`ere mass in $\bC^2$ as a K\"ahler cone.
Recall that we have introduced the notation for a positive real number $A$ as 
$$L_A(u) =  || ru_r ||_{L^{\infty}(\ol{B}_R)}, $$
where $R = e^{-A}$. 
\begin{theorem}
\label{thm-mt-000}
For a function $u$ in $\cP^{2}(B_1)$ and an $R\in (0,1)$, 
we have the following estimate: 
\begin{equation}
\label{sub-0081}
\frac{1}{\pi^2}\int_{B_R} (dd^c u)^2 \leq 4 L_A(u)\cdot \left( \fint_{S_R} u_t \right), 
\end{equation}
where $t=\log r $ and  $A= -\log R$. 
\end{theorem}
\begin{proof}
Apply equation \eqref{ld-005} and we obtain the estimate on the first term of its R.H.S. as  
\begin{eqnarray}
\label{sub-0085}
&& \int_{\bC\bP^1} \left( \int_{S^1} u_{t} \left\{ X_1 \bar X_1 u +  r(\bar X_0 u) h_{\z\bar\z} \right\} d\theta \right) i d\z \wedge d\bar\z
\nonumber\\
&\leq&  L_A(u) \int_{\bC\bP^1} 
\left( \int_{S^1} \left\{ u_{ \z\bar\z} + \frac{1}{2} u_{ t} h_{\z\bar\z}  \right\} d\theta \right) i d\z \wedge d\bar\z
\nonumber\\
&=&  L_A(u) \int_{\bC\bP^1} \left( \int_{S^1} u_{ t} d\theta \right) \omega_{FS}
\nonumber\\
&\leq&  4\pi^2 L_A(u) \cdot  \left( \fint_{S_R} u_t \right). 
\end{eqnarray}
Here we have used  
equation \eqref{ld-001} and its positivity on the second line of the above inequality. 
Fubini's Theorem is utilized on the third line. 

Thanks to the Pohozaev inequality (Lemma \ref{lem-sub-002}), 
we can estimate the second term as 
\begin{eqnarray}
\label{sub-0086}
&&\int_{\bC\bP^1} \left| \int_{S^1} \left( |u_t|^2 - 4 | u_{\theta}|^2 \right) d\theta\right| \omega_{FS} 
\nonumber\\
&\leq& 2 L_A(u)  \int_{\bC\bP^1} \left( \int_{S^1} u_{ t} d\theta \right) \omega_{FS}
\nonumber\\
& = & 8\pi^2 L_A(u) \cdot  \left( \fint_{S_R} u_t \right). 
\end{eqnarray}
Combining with equations \eqref{ld-005}, \eqref{sub-0085} and \eqref{sub-0086}, 
our result follows. 

\end{proof}

We emphasize that Theorem \ref{thm-mt-000} in fact gives an $L^1$-apriori estimate 
on the complex Monge-Amp\`ere operator on the K\"ahler cone. 
That is, the Monge-Amp\`ere mass on a ball can be uniformly estimated by 
the supreme of the $t$-derivative of a plurisubharmonic function. 
Furthermore, 
it reveals an interesting geometric property on the cone in the following sense. 

Suppose $u$ is a smooth strictly plurisubharmonic function on $\bC^2$.
Thus $\omega_u: = dd^c u$ is a K\"ahler metric on the cone $\bC^2 - \{ 0 \}$. 
In particular, if we take $u_0(z) = \frac{1}{2} || z ||^2$ then the standard Euclidean metric on $\bC^2$ is obtained as 
$$ dd^c u_0 =  \frac{i}{2} dz^0 \wedge d\bar z^0 + \frac{i}{2}dz^1 \wedge d\bar z^1.$$
In this point of view, Theorem \ref{thm-mt-000} gives a volume comparison between 
the Monge-Amp\`ere measure and the Euclidean volume of a ball.

\begin{cor}
\label{cor-mt-000}
Let $\omega_u$ be a K\"ahler metric on the cone $\bC^2-\{0 \}$. 
Then we have following volume comparison estimate: 
\begin{equation}
\label{sub-0087}
\frac{\emph{MA}(u)(B_R)}{\Vol(B_R)}     = \frac{\int_{B_R} \omega_u^2}{ \int_{B_R} \omega_{u_0}^2}
\leq  8 R^{-3} || ru_r ||_{L^{\infty}(\ol{B}_R)}\cdot \left( \fint_{S_R} u_r \right),
\end{equation}
for all $R\in (0, \infty)$. 
\end{cor}


As an application of this comparison theorem, 
we consider the case when the ball is shrinking to the vertex of the cone, i.e. $R\rightarrow 0$. 
Then it gives the following upper bounded estimate on the residual Monge-Amp\`ere mass. 
(Instead of the Euclidean metric, we take $u_1 (z): = \log ||z||  $ in the denominator of the comparison theorem.) 

\begin{theorem}
\label{thm-mt-001}
Suppose a function $u\in\cP(B_1)$ is uniformly directional Lipschitz continuous 
near the origin. 
Then we have the following estimate: 
$$ \tau_u(0) \leq 4 \k_u(0) \cdot \nu_u(0). $$
\end{theorem}
\begin{proof}
Let a function $u\in \cP(B_1)$ be uniformly directional Lipschitz continuous near the origin.  
For a small $\delta$ fixed, 
we consider the sequence of its standard regularizations $u_{\ep}: = u * \rho_{\ep}$, 
and then it is in the family $\cP^{2}(B_{1-\delta})$ for each $0<\ep<\delta$. 

Thanks to the estimate in Theorem \ref{thm-mt-000}, 
we obtain the following estimate for each $\ep$ and $R= e^{-A}$:  
\begin{equation}
\label{sub-011}
\int_{S_R} d^c u_{\ep} \wedge dd^c u_{\ep} 
\leq  4\pi^2 L_A(u_{\ep})\cdot \left( \fint_{S_R} u_{\ep, t} \right).
\end{equation}

Let $\ep$ approach zero.
Then it follows from 
Lemma \ref{lem-mt-001} and \ref{lem-mt-002} 
that the L.H.S. of equation \eqref{sub-011} converges 
to $\int_{B_R}(dd^c u)^2$ for almost all radius $R$.  
Furthermore, 
Lemma \ref{lem-mt-004} and \ref{lem-mt-005} 
imply that the R.H.S. is controlled by 
$$ 4\pi^2 L_A(u) \cdot \left( \fint_{S_R} u_{t} \right) + C\ep, $$
for some uniform constant $C>0$, 
and hence we obtain 
\begin{equation}
\label{sub-012}
\pi^{-2} \cdot \mbox{MA}(u)(B_R) \leq 4 L_A(u) \cdot \left( \fint_{S_R} u_t \right), 
\end{equation}
for almost all $R\in (0,1)$. 
Observe that the Monge-Amp\`ere mass of $u$ is non-decreasing in $R$,
and then our result follows by taking $A\rightarrow +\infty$ or $R\rightarrow 0$. 
 
\end{proof}

Hence the zero mass conjecture follows under the same condition. 

\begin{cor}
\label{cor-mt-001}
Suppose a function $u\in\cP(B_1)$ is uniformly directional Lipschitz continuous near the origin. 
Then we have 
$$ \nu_u(0) =0 \Rightarrow \tau_u(0) =0.  $$
\end{cor}

\subsection{Fourier expansions}
The example in equation \eqref{ex-001} shows that 
the uniform directional Lipschitz continuity may fail in general.  
Then one possible approach 
is to consider the Fourier expansion of the alternating part 
of $u$ along the $S^1$-direction, i.e. we try to write $v:= u-u_s$ 
under a complex Hopf-coordinate as 
$$ v \sim \sum_{k=1}^{\infty} \left( \cos k\theta \cdot v_k  +   \sin k\theta \cdot w_k  \right), $$
where $v_k$ and $w_k$ are functions of $(r, \z, \bar\z)$ only. 
In other words, the $S^1$-symmetrization $u_s$ is exactly the constant term 
in the Fourier series. 
 
However, 
this Fourier expansion rather depends on the chosen local coordinate, 
and it may vary under different local embeddings. 
For instance, we can consider the following function $u\in \cP(B_1)$ in $\bC^2$: 
\begin{equation}
\label{fe-0001}
u(z^0, z^1): = u_s(z^0, z^1) +  \Re z^1,  
\end{equation}
where $u_s$ is any function in $\cF(B_1)$. 

On the one hand, 
the complex Hopf-coordinate that we usually mentioned 
actually corresponds to the local embedding: 
\begin{equation}
\label{fe-000}
 (r, \theta, \z, \bar\z) \rightarrow \left( \frac{  r e^{\frac{i}{2}\theta } ( |\z| / \z)^{1/2}}{(1+ |\z|^2)^{1/2}},  \ \
\frac{  r e^{\frac{i}{2}\theta } (|\z| \cdot \z)^{1/2}}{(1+ |\z|^2)^{1/2}}  \right), 
\end{equation}
and then the expansion of the alternating part of $u$ is  
\begin{eqnarray}
\label{fe-001}
v(z) =  \cos\left(\frac{\theta}{2}\right) \left\{ r  \cos\left(\frac{\vp}{2}\right)  \sin\left(\frac{\eta}{2}\right)  \right\}
- \sin\left(\frac{\theta}{2}\right) \left\{ r  \sin\left(\frac{\vp}{2}\right)  \sin\left(\frac{\eta}{2}\right)  \right\}, 
\nonumber\\
\end{eqnarray}
where $\vp, \eta$ are the Euler angles of $S^2$ and we have 
$ \z = \tan(\eta/2)e^{i\vp}$. 

On the other hand, 
there is another often used local embedding as 
\begin{equation}
\label{fe-0011}
(r, \theta', \z, \bar\z) \rightarrow \left( \frac{re^{i\theta'} \z}{(1+ |\z|^2)^{1/2}},  \frac{re^{i\theta'}}{(1+ |\z|^2)^{1/2} }  \right). 
\end{equation}
Then the Fourier expansion under this embedding is simply  
\begin{equation}
\label{fe-002}
v(z) =  \frac{r}{(1+ |\z|^2)^{1/2}} \cos\theta'. 
\end{equation}

Therefore, we need to clarify a ``nice'' complex Hopf-coordinate, 
before discussing about the Fourier expansion. 

Let $p$ denote the fiber map of the circle bundle $ S^1\hookrightarrow S^3\xrightarrow{p} S^2 $.
Then a complex Hopf-coordinate  means that 
there exists an open set $U\subset \bC\bP^1$,
and a local trivialization of the fiber map $p$ as 
$$ \phi: S^1\times U \rightarrow p^{-1} (U) \subset S^3, $$
such that the induced map  
$$ \Phi: \bR_+ \times S^1\times U \rightarrow (\bC^2)^* $$
defined as $\Phi (r, \theta, \z, \bar\z): = (r, \phi(\theta, \z, \bar\z))$ is a local embedding. 
Moreover, we say that this complex Hopf-coordinate is \emph{proper}
if the the complement of $U$ has measure zero in $\bC\bP^1$. 

For instance, the two local embeddings in equation \eqref{fe-000} and \eqref{fe-0011}
both induce proper complex Hopf-coordinates. 
Inspired from the example in equation \eqref{fe-002},
we introduce the following concept as a first attempt.

Let $f(\theta)$ be a continuous function on the unit circle $S^1$ that its average vanishes, 
and we write 
\begin{equation}
\label{fe-003}
\sup_{\theta} f = K,\ \ \ \mbox{and} \ \ \ \inf_{\theta} f= -k. 
\end{equation}
 
\begin{defn}
\label{defn-sf-001}
For a function $u\in \cP(B_1)$,
we say that $u$ separates the $S^1$-direction in its alternating part, 
if we can write $u$ near the origin under 
a proper complex Hopf-coordinate as 
$$ u: = u_s + f(\theta) v$$
where $v: = v(r, \z, \bar\z)$ is a circular symmetric $W^{1,2}_{loc}$-function.
\end{defn}

The upshot is that this kind of functions
in fact satisfy our uniform continuity condition, 
and the zero mass conjecture follows. 

\begin{theorem}
\label{cor-sf-001}
Suppose a function $u\in \cP(B_1)$ separates the $S^1$-direction in its alternating part. 
Then it is uniformly directional Lipschitz continuous near the origin, 
and we have the estimate
$$  \tau_u(0) \leq 4 \k_u(0)\cdot \nu_u(0).  $$
In particular, the zero mass conjecture follows in this case. 
\end{theorem}
\begin{proof}
We note that the symmetric part $u_s$ always has uniformly bounded $t$-derivative. 
Then it suffices to prove that the first derivative  
$ \d_t v$ is in fact an $L^{\infty}$-function in a smaller ball. 

To this end, 
we fix a complex direction $\z$ that is in the defining area of 
the proper complex Hopf-coordinate. 
Then we can write the alternating part for almost all points $\z\in \bC\bP^1$ as 
$$ u-u_s = f(\theta)v.$$

If the function $f(\theta)$ 
is identically zero, then we have $u =u_s$ almost everywhere. 
This means that we are back to the $S^1$-invariant case. 

Hence we can assume that the function $f(\theta)$
is non-trivial. 
It follows that 
the two constants $K$ and $k$ must be positive. 
Take a new constant as 
$c: = (K-k)/2$, 
and we denote $g$ by the following function
$$ g(\theta): = f(\theta) -c, $$
and then we have 
\begin{equation}
\label{sf-0010}
\sup_{\theta} g = -\inf_{\theta} g = (K+k)/2.
\end{equation}
Hence the alternating part can be rewritten as 
$$ u - u_s = c v + g(\theta)v,$$
where $u_s$ and $v$ are both $S^1$-invariant functions. 
 
Next we take the difference between the supreme of $u$ and its average along the circle. 
Thanks to equation \eqref{fe-003} and \eqref{sf-0010}, we have  
\begin{equation}
\label{sf-001}
\sup_{S^1} u - u_s = c v + \sup_{\theta } \left\{ v\cdot g(\theta)  \right\} = \frac{1}{2}(K+k) |v| + \frac{1}{2}(K-k) v.
\end{equation}
Recall that we denoted $\hat u$ by the supreme as  
$$ \hat u(t, \z) = \sup_{\theta} u(e^t, \theta, \z, \bar\z). $$
Thanks to Lemma \ref{lem-gc-001}, we obtain
\begin{equation}
\label{sf-002}
- M_A (u) \leq (\d_t^+ \hat u - \d_t^+ u_s )|_{t\leq -A} \leq N_A (u),
\end{equation}
since the numbers $M_A$ and $N_A$ are both non-increasing in $A$.


Then we note that the functions $v$ and hence $|v|$ are
 in the $W^{1,2}_{loc}(B_1)$ Sobolev space, see Lemma 7.6 in \cite{GT}. 
This further implies that their first derivatives are $L^2$-functions, 
and exist in the classical sense for almost all points. 
Therefore, the following equation holds for almost all points: 
$$   \frac{1}{2} (K+k) \d_t |v| + \frac{1}{2} (K-k)\d_t  v =     
\left\{ \begin{array}{rcl}
        K v_t & \mbox{for} &  v>0, \\
         0  & \mbox{for} &  v=0, \\
         -k v_t & \mbox{for} & v<0.
                \end{array}\right. $$
It follows that we have the bound almost everywhere as 

\begin{equation}
\label{sf-0030}
-   \max \left\{ \frac{M_A(u)}{K}, \frac{N_A(u)}{k} \right\} \leq   r\d_r v \leq \max \left\{ \frac{N_A(u)}{K}, \frac{M_A(u)}{k} \right\}. 
\end{equation}

In other words, 
we have proved that the $L^2$-function $r\d_r v$ is in fact uniformly bounded 
almost everywhere, 
and hence is an $L^{\infty}$-function in a ball. 
Hence our result follows. 


\end{proof}

\section{Remarks}
\smallskip

In this section, 
we are going to make several remarks to 
reveal further difficulties along this approach. 
Moreover, a conjecture about plurisubharmonic functions with finite frequencies will be given, 
perhaps as a middle step towards the zero mass conjecture.

\subsection{Examples}
In Example \ref{exm-001},
we have seen a function in the family $\cP(B_1)$ that is not uniformly directional Lipschitz continuous. 
In the following, another such example will be provided in $\bC^2$, even if it has certain symmetry, see \cite{CG09}. 

\begin{example}
\label{exm-002}
Let $z:= (z^0, z^1)$ be the complex Euclidean coordinate in $\bC^2$, 
and then we consider the following function for any integer $n\geq2$:
\begin{equation}
\label{ex-010}
u(z):= \frac{1}{2n} \log\left( |z^1 - (z^0)^n |^2 + |z^1|^{2n} \right). 
\end{equation}
\end{example}
It is clear that this function is in the family $\cP(B_1)$, 
and it has the so called $(1,n)$-symmetry. 
Moreover, 
it can be transformed into a circular symmetric function, 
after passing to a finite holomorphic cover near the origin, see \cite{Li23}.
However, we claim that this function fails to be uniformly directional Lipschitz continuous near the origin. 
This shows that the Sasakian structure is not fit to
the change of holomorphic coordinates. 

Consider the following complex Hopf-coordinate: 
$$ z^0: = \frac{ r e^{i\theta}}{(1+ |\xi|^2)^{1/2}}, \ \ \  z^1: =  \frac{ r e^{i\theta} \cdot \xi }{(1+ |\xi|^2)^{1/2}},$$
where $\xi: = z^1 / z^0 $ is a holomorphic coordinate on the open set $\{z^0 \neq 0 \}$. 
Then we further compute as 
\begin{eqnarray}
\label{ex-011}
z^1- (z^0)^n &=& \frac{ r e^{i\theta} \cdot \xi }{(1+ |\xi|^2)^{1/2}} - \frac{ r^n e^{in\theta}  }{(1+ |\xi|^2)^{n/2}}
\nonumber\\
&=&  re^{i\theta} \cdot   \frac{ \xi (1+ |\xi|^2)^{\frac{n-1}{2}} - r^{n-1} e^{i(n-1)\theta} }{ (1 + |\xi|^2)^{n/2}}. 
\end{eqnarray}
It follows that we can infer the equation: 
\begin{eqnarray}
\label{ex-012}
&& \left| z^1- (z^0)^n \right|^2 + |z^1|^{2n}  
\nonumber\\ 
&=& r^2 \frac{|\xi|^2(1+ |\xi|^2)^{n-1}  + r^{2n-2}(1+ |\xi|^{2n})  - 2 r^{n-1}\Re\{e^{-i(n-1)\theta}\xi \} (1+ |\xi|^2)^{\frac{n-1}{2}} }{(1+ |\xi|^2)^n}.
\nonumber\\
\end{eqnarray}
Moreover, we have 
\begin{eqnarray}
\label{ex-013}
&& \frac{1}{2n-2} \left( 2nru_r -2 \right)  
\nonumber\\ 
&=& \frac{r^{2n-2}(1+ |\xi|^{2n}) -  r^{n-1} \Re\{e^{-i(n-1)\theta}\xi \} (1+ |\xi|^2)^{\frac{n-1}{2}}  }
{|\xi|^2(1+ |\xi|^2)^{n-1}  + r^{2n-2}(1+ |\xi|^{2n})  - 2 r^{n-1}\Re\{e^{-i(n-1)\theta}\xi \} (1+ |\xi|^2)^{\frac{n-1}{2}} }.
\nonumber\\
\end{eqnarray}
If we further take the following two complex numbers 
$$a: = \xi (1+ |\xi|^2)^{\frac{n-1}{2}}, \ \ \ b:= r^{n-1} e^{i(n-1)\theta}, $$
then equation \eqref{ex-013} can be re-written as  
\begin{equation}
\label{ex-014}
 1+ \frac{\Re(a\bar b) - |a|^2}{| a- b|^2 + r^{2n-2}|\xi|^{2n}}.
\end{equation}
Pick an angle $\theta$ such that we have $\Re(a\bar b)= |a||b|$, 
and then it boils down to estimate the term 
\begin{equation}
\label{ex-015}
\frac{|a| (r^{n-1}-|a|)}{ \left(r^{n-1} - |a| \right)^2 + r^{2n-2}|\xi|^{2n}}.
\end{equation}
However, we can choose $r$ and $\xi$ small enough such that $|a| = (1- k^{-1})r^{n-1}$ for an integer $k>0$,
and equation \eqref{ex-015} further reduces to  
\begin{equation}
\label{ex-016}
\frac{ (k-1) r^{2n-2}}{  r^{2n-2} + k^2 O(r^{2n^2-2})} = \frac{k-1}{1+ k^2 O(r^{2n^2 -2n})} \geq \frac{k-1}{2},
\end{equation}
for all $r$ small enough. 
Then it is clear that equation \eqref{ex-016} can be arbitrarily large, 
and our claim follows.

\subsection{Higher dimensional cases}
\label{sec-61}

Next we would like to discuss some difficulties on the estimate in higher dimensions, e.g. in $\bC^3$. 
Consider the decomposition formula in Theorem \ref{thm-df-001}, 
and then we have the term for $k=2$ as 
$$ ( 2rX_0 u )^3 + (2r \bar X_0 u)^3 =  2 (u^3_t - 12 u_t u_{\theta}^2 ).$$
Hence we can still use the Pohozaev inequality to estimate this term, 
since the order of $u_{\theta}$ is even and we already have $L^{\infty}$-bound on $u_t$. 

However, we have the following mixed term for $k=1$ as 
\begin{equation}
\label{rem-001}
\int_{S_r} \left( 4|u_{\theta}|^2 - |u_t|^2 \right) \omega \wedge \Theta \wedge d\theta.
\end{equation}
Unfortunately, there is no way to use the Pohozaev inequality, 
since the $(1,1)$-form $\Theta$ involves the $t$-variable in a complicated way.
Therefore, we need to consider to add one more condition on the uniform bound of the derivative $| u_{\theta} |$ at this time.

Moreover, the last term is the integral for $k=0$ as 
\begin{equation}
\label{rem-002}
\int_{S_r} (ru_r) \Theta^2 \wedge d\theta. 
\end{equation}
Here we still need to invoke the positivity of $\Theta$ to perform estimates. 
However, the integration by parts on the $S^1$-direction for the form $\Theta^2$ is no longer easy, 
since it involves quadratic terms of the second derivatives of $u$. 

In conclusion, 
our computations indicates that it would be easier 
to find counter-examples to the conjecture in higher dimensions, 
if there is one. 

\subsection{Finite frequencies}
Finally, we put some remarks on a plurisubharmoinc function and its Fourier expansions. 
Like in Definition \ref{defn-sf-001}, 
we say that a function $u\in \cP(B_1)$ has finite frequencies in its alternating part,  
if it can be written near the origin under a proper complex Hopf-coordinate as 
$$ u : = u_s +\sum_{k=1}^m \left( \cos k\theta \cdot v_k + \sin k\theta \cdot w_k \right). $$

As we have seen from the above,
there are plenty of such examples. 
Inspired from Corollary \ref{cor-sf-001}, 
we would like to arise a weaker version of the zero mass conjecture.

\begin{conj}
\label{conj-001}
Suppose a function $u\in \cP(B_1)$ has finite frequencies in its alternating part. 
Then we still have an upper-bound estimate on its Monge-Amp\`ere mass, and it follows
$$ \nu_u(0) =0 \Rightarrow \tau_u(0) =0.  $$
\end{conj}

\bigskip

\section{Appendix A: Plurisubharmonicity and Sasakian geometry}
\smallskip

There is a natural K\"ahler cone structure on the space $(\bC^{n+1})^*\cong (\bR^{2n+2})^*$
that induces the standard Sasakian structure of the unit sphere $S^{2n+1}$. 
The standard almost complex structure $I$ on $\bR^{2n+2}$ is a $(1,1)$-tensor field as 
$$ I: = \sum_A \left(  \frac{\d}{\d y^A} \otimes dx^A    - \frac{\d}{\d x^A} \otimes dy^A  \right).  $$
Then the \emph{Reeb vector field} is defined as 
$$ \xi_0:= - I (r\d_r), $$
and its dual is the so called \emph{contact 1-form} as 
$$ \eta_0: = I (r^{-1} dr).  $$
Moreover, we also take its normalization as 
$ \eta: = -\eta_0/2$.
There is another $(1,1)$-tensor field $\Phi_0$ that plays the role of the almost complex structure of the cone as
\begin{eqnarray}
\label{ss-001}
\Phi_0: &=& I - (r\d_r) \otimes \eta_0.  
\end{eqnarray}
Then it is clear to see $\Phi_0 (\xi_0) =0$. 
Moreover, the tangent space of the unit sphere can be split as 
$$ TS^{2n+1} = L_{\xi_0} \oplus \cD, $$
where $L_{\xi_0}$ is the trivial line bundle generated by $\xi_0$,  
and $\cD: = \ker (\eta_0)$ is the so called \emph{contact sub-bundle}. 
Then the restriction of $\Phi_0$ to the tangent space $TS^{2n+1}$ gives the following relation: 
$$ \Phi_0^2 = -\mathbb{I} + \xi_0 \otimes \eta_0. $$ 
Combining with the compatible metric cone structure, 
it follows that the \emph{almost CR-structure} $(\cD, \Phi_0|_{\cD})$
is integrable, and we have a splitting as 
\begin{equation}
\label{ss-002}
\cD \otimes  \bC = \cD^{1,0} \oplus \cD^{0,1} \ \ \ \text{with} \ \ \   \overline{\cD^{1,0}} = \cD^{0,1},
\end{equation}
where $\cD^{1,0}$ and $D^{0,1}$ are eigenspaces of $\Phi_0$
with respect to the eigenvalues $i$ and $-i$, respectively. 
Here the integrability condition means the following relation:
$$ \left[ \cD^{1,0}, \cD^{1,0}  \right] \subset \cD^{1,0} \ \ \ \text{and} \ \ \  \left[ \cD^{0,1}, \cD^{0,1}  \right] \subset \cD^{0,1}.$$

For more details about Sasakian geometry, 
the readers are referred to the book by Boyer and Galicki, \cite{BG08}.

\subsection{Local frames and coframes}
The complex Hopf-coordinate on the K\"ahler cone $(\bC^{n+1})^*$ is introduced in \cite{HLX}, 
and we recall it here. 
Consider the holomorphic functions $\z^{\a}: = z^{\a}/ z^0$ on the set $\{ z^0 \neq 0 \}$ for all 
$\a = 1, \cdots, n$. 
Let $r = |z|$ be the radius function, and then a real coordinate system on $(\bR^{2n+2})^*$
can be written as 
$$ (r, \theta, \z, \bar\z): = (r, \theta, \z^1, \cdots, \z^n, \bar\z^1, \cdots, \bar\z^n),  $$
for all $r\in \bR_+, \theta\in \bR$ and $\z^{\a}\in \bC$. 
Here the complex variable $\z$ can be thought of as a point on the projective space $\bC\bP^n$
via the Hopf-fibration. 
Moreover, we have used the following embedding: 
\begin{equation}
\label{bf-001}
z^0: = r e^{\frac{i}{2}\theta} \frac{\varrho(\z ,\bar\z)}{ \left( 1+ \sum_{\b} |\z^{\b}|^2 \right)^{1/2}}; \ \ \ 
z^{\a}: = r e^{\frac{i}{2}\theta} \frac{\z^{\a} \cdot \varrho(\z ,\bar\z)}{ \left( 1+ \sum_{\b} |\z^{\b}|^2 \right)^{1/2}}, 
\end{equation}
where the factor is 
$$ \varrho(\z, \bar\z):= \prod_{\a=1}^n \left( \bar\z^{\a} |\z^{\a}|^{-1}\right)^{\frac{1}{2}}. $$
Then we can introduce a corresponding \emph{local basic function} as 
\begin{equation}
\label{bf-002}
h(\z, \bar\z): = \log\left( 1 + \sum_{\a} |\z^{\a}|^2 \right) - \sum_{\a} \log |\z^{\a}|, 
\end{equation}
and we can also write it as 
\begin{equation}
\label{bf-003}
h(z,\bar z) = \log\left(  \sum_{A} |z^{A}|^2 \right) -\sum_{A} \log |z^{A}|. 
\end{equation}
Furthermore, we obtain a local defining equation of the contact $1$-form as 
\begin{eqnarray}
\label{bf-004}
\eta &=& \frac{1}{4} \left\{ d\theta - i (\d_{\z} h - \dbar_{\z} h )    \right\}
\nonumber\\
&=&   \frac{1}{4} \left\{ d\theta -  \sum_{\a} \cos \k_{\a} \cdot \Im \left( \frac{d\z^{\a}}{\z^{\a}} \right) \right\},
\end{eqnarray}
where we denote $\d_{\z}: = \sum \d / \d \z^{\a}, \dbar_{\z}: = \sum \d/ \d \bar\z^{\a}$ 
for all $\a =1,\cdots, n$ and 
$$ \cos\k_{\a}: = 1- \frac{2|\z^\a|^2}{1+ \sum_{\b} |\z^\b|^2}. $$
In other words, we have 
\begin{equation}
\label{bf-005}
\eta_0 = - \left( d\theta/2 + d^c_{\z}h \right).
\end{equation}
Then it follows 
\begin{equation}
\label{bf-006}
\xi_0 = -2 \d_{\theta},
\end{equation}
and 
\begin{equation}
\label{bf-007}
d\eta = \frac{1}{2} dd^c h = \omega_{FS},
\end{equation}
where $\omega_{FS}$ is the Fubini-Study metric on $\bC\bP^n$ with total volume $\pi^n$. 

Next we are going to construct the standard frame fields with respect to the K\"ahler cone structure. 
First, along the complex line $L_{\z}$ through the origin in a fixed direction $\z$, we define 
\begin{eqnarray}
\label{bf-008}
X_0:&=& \frac{1}{2} \left( \d_r + i r^{-1} \xi_0 \right) =  \frac{1}{2} \left( \d_r - 2i r^{-1} \d_{\theta} \right);
\\
\bar X_{0}:&=& \frac{1}{2} \left( \d_r - i r^{-1} \xi_0 \right)=  \frac{1}{2} \left( \d_r + 2i r^{-1} \d_{\theta} \right). 
\end{eqnarray}
Moreover, on the complexified contact sub-bundle $\cD\otimes \bC$, we have the following frame equations 
corresponding to the splitting in equation \eqref{ss-002}:
\begin{eqnarray}
\label{bf-009}
X_{\a}:&=& \d_{\a} + \frac{i}{4\z^{\a}} \cos\k_{\a} \cdot \xi_0 = \d_{\a} + i ( \d_{\a} h) \d_{\theta};
\\
\bar X_{\a}:&=& \d_{\bar\a} - \frac{i}{4\bar\z^{\a}} \cos\k_{\a} \cdot \xi_0= \d_{\bar\a} - i ( \d_{\bar\a} h) \d_{\theta},
\end{eqnarray}
where the operator $\d_{\a}$ stands for $\frac{\d}{\d \z^{\a}}$ for all $\a = 1, \cdots, n$. 
Then it is standard to check that we have 
$$ I(X_A) = i X_A; \ \ \ I(\bar X_A) = -i \bar{X}_A, $$
for all $A= 0, 1, \cdots, n$. 
In conclusion, 
the following vector fields build a local frame of the tangent space $T(\bR^{2n+2})^*$ as 
$$ \{ X_0, X_{\a}, \bar X_{0}, \bar X_\a \}.$$ 

Moreover, its dual basis in fact builds a coframe as follows:   
\begin{eqnarray}
\label{bf-011}
&& \lambda^0:= dr - ir\eta_0; \ \ \ \bar\lambda^0 : = dr + ir \eta_0; 
\\
&& \lambda^{\a}: = d\z^{\a}; \ \ \ \ \ \ \ \ \ \ \bar\lambda^{\a}: = d\bar\z^{\a};
\end{eqnarray}
Then the $(1,1)$-tensor field $\Phi_0$ restricted to the sphere can be re-written as 
$$ \Phi_0: = i\sum_{\a} X_{\a}\otimes d\z^{\a} - i \sum_{\a} \bar X_{\a} \otimes d\bar\z^{\a}, $$
since we have 
\begin{equation}
\label{bf-012}
\Phi_0(X_\a) = i X_\a; \ \ \ \Phi_0(\bar X_\a) =- i \bar X_\a.
\end{equation}
After a standard computation, we come up with the following commutators: 
\begin{eqnarray}
\label{bf-0130}
&& [X_\a, X_\b] = [\bar X_\a, \bar X_{\b}]=0;
\\
\label{bf-0131}
&& [X_\a,  \d_{r}] = [\bar X_\a, \d_{r}]= [X_\a,  \d_{\theta}] = [\bar X_\a, \d_{\theta}]=0;
\\
\label{bf-0132}
&& [X_\a, \bar X_\b] = i h_{\a\bar\b} \xi_0;
\\
\label{bf-0133}
&& [X_\a, X_0] = [\bar X_\a, X_{0}]=[X_\a, \bar X_0] = [\bar X_\a, \bar X_{0}]= 0;
\\
\label{bf-0134}
&& [X_0, \bar X_0] = -i r^{-2} \d_{\theta}.
\end{eqnarray}

Finally, 
we note that
the vector fields $\{ \xi_0, X_{\a}, \bar X_{\a} \}$
build a frame of the tangent bundle $TS^{2n+1}$
with its dual basis $\{ \eta_0, \lambda^{\a}, \bar\lambda^{\a} \}$. 

\begin{rem}
It is worthy mentioning that the above frames and coframes 
actually come from the Sasakian structure. 
In other words, the same equations will hold under the vary of local coordinates,
possibly with a different basic function $h$ and a renormalizaed angle $\theta$. 
For more details about Sasakian manifolds, see \cite{GKN00}, \cite{GZ12}, \cite{HL21}, \cite{H13} and \cite{HS16}.
\end{rem}

\subsection{The complex Hessian}
Next we are going to compute the complex Hessian of 
a function $u\in \cP^{2}(B_1)$ with respect to the Sasakian structure. 
First, we compute its first derivatives as follows: 
\begin{eqnarray}
\label{ch-001}
\d u &=& (X_0 u) \lambda^0 + (X_\a u) d\z^{\a};
\\
\label{ch-002}
\dbar u &=&  (\bar X_0 u) \bar{\lambda}^0 + (\bar X_\a u) d\bar\z^{\a}.
\end{eqnarray}
Recall the convention $d^c: = \frac{i}{2} (\dbar - \d)$
and hence we have 
\begin{equation}
\label{ch-003}
d^c u = \frac{i}{2}\left\{  (\bar X_0 u) \bar\lambda^0 - (X_0 u) \lambda^0  \right\} + \frac{i}{2}\left\{(\bar X_\a u) d\bar\z^{\a}   - (X_\a u) d\z^{\a} \right\}.
\end{equation}
Equipped with equation \eqref{bf-007}, \eqref{bf-008} and \eqref{bf-009}, 
we obtain 
\begin{equation}
\label{ch-004}
d^c u = u_t \eta - u_{\theta} dt + \frac{1}{2} u_{\theta} dh + d^c_{\z} u,
\end{equation}
where $t=\log r$ and the operator $d_{\z}$ is defined by 
$$  d_{\z}: = \d_{\z} + \dbar_{\z}. $$
Before moving on, 
we need to perform the following computations:  
\begin{equation}
\label{ch-005}
\d \bar \lambda^0 = \d dr  + i \d r \wedge \eta_0 + i r \d \eta_0,
\end{equation}
where 
$$ \d \eta_0 = - \frac{1}{2} (i \ddbar_{\z} h), \ \ \ \mbox{and} \ \ \  \d r = \frac{1}{2} \lambda^0.  $$
Hence we have 
\begin{equation}
\label{ch-006}
\d \bar \lambda^0 = \frac{1}{2r} \lambda^0 \wedge \bar\lambda^0 + r\ddbar_{\z} h.
\end{equation}
Finally, we obtain
\begin{eqnarray}
\label{ch-007}
\ddbar u &=& \d (\bar X_0 u) \wedge \bar\lambda^0 + (\bar X_0 u) \d \bar\lambda^0 + \d (\bar X_{\b} u) \wedge d\bar\z^{\b}
\nonumber\\
&=& \left\{ X_0 \bar X_0 u + \frac{1}{2r} \bar X_0 u \right\} \lambda^0 \wedge \bar\lambda^0
\nonumber\\
&+& \left( X_{\a} \bar X_0 u \right) d\z^{\a} \wedge \bar\lambda^0 + \left( X_0 \bar X_{\a} u \right) \lambda^0 \wedge d \bar\z^{\a}
\nonumber\\
&+& \left\{ X_{\a}\bar X_{\b} u  + r (\bar X_0 u) h_{\a\bar\b} \right\} d\z^{\a} \wedge d\bar\z^{\b}. 
\end{eqnarray}
Thanks to the commutators in equations \eqref{bf-0130} - \eqref{bf-0134}, 
we can further check the following equation:
\begin{eqnarray}
\label{ch-008}
\ddbar u &=& - \bar\d \d u 
\nonumber\\
&=& \left\{ \bar X_0  X_0 u + \frac{1}{2r}  X_0 u \right\} \lambda^0 \wedge \bar\lambda^0
\nonumber\\
&+& \left( \bar X_{\a} X_0 u \right) \lambda^0 \wedge d\bar\z^{\a} + \left( \bar X_0 X_{\a} u \right) d\z^{\a} \wedge \bar\lambda^0
\nonumber\\
&+& \left\{ \bar X_{\b} X_{\a} u  + r (X_0 u) h_{\a\bar\b} \right\} d\z^{\a} \wedge d\bar\z^{\b}. 
\end{eqnarray}
This formula is in fact well known among experts, 
and it is 
the complex Hessian of $u$ under the Sasakian structure.  
Here we use it to give a new description of 
the plurisubharmonicity of a function. 
First we define the following two matrices: 
\\
$$
S(u): = 
\begin{pmatrix}
X_0 \bar X_0 u +  (2r)^{-1} \bar X_0 u & X_0 \bar X_{\b} u \\
\\
X_{\a} \bar X_0 u &  X_{\a}\bar X_{\b} u + r (\bar X_0 u) h_{\a\bar\b}
\end{pmatrix}, 
$$
and 
$$
S^*(u): = 
\begin{pmatrix}
\bar X_0  X_0 u +  (2r)^{-1}  X_0 u & \bar X_{\b} X_0 u \\
\\
 \bar X_0 X_{\a} u &  \bar X_{\b} X_{\a} u + r ( X_0 u) h_{\a\bar\b}
\end{pmatrix}. 
$$
\\
Thanks to equation \eqref{ch-007} and \eqref{ch-008}, 
these two matrices are actually equal, i.e. we have proved that 
$$S(u) = S^*(u),$$ 
and then our positivity result follows.

\begin{prop}
\label{prop-ch-001}
Suppose $u$ is 
a $C^2$-continuous function on a domain $D \subset (\bC^{n+1})^*$. 
Then it is plurisubharmonic, 
if and only if the matrix 
$S(u)$, or its different version $S^*(u)$, is everywhere semi-positive.
Moreover,
it is strictly plurisubharmonic, if and only if the matrix is positive on $D$. 
\end{prop}

Let us take a closer look at the matrix $S(u)$. 
First, the $(1,1)$-entry of this matrix 
is exactly the Laplacian of the restriction of $u$ 
to the complex line $\l_{\z}$ 
through the origin in a direction $\z\in \bC\bP^n$, 
i.e. we have 
\begin{eqnarray}
\label{ch-009}
&&4 (X_0\bar X_0 u + \frac{1}{2r} \bar X_0 u )
\nonumber\\
&=&  u_{rr} + r^{-1} u_r + 4 r^{-2} u_{\theta\theta} 
\nonumber\\
&=& \Delta ( u|_{\l_{\z}}) \geq 0. 
\end{eqnarray}
Then 
we compute the $(n\times n)$-submatrix on the lower-right corner 
  in the complex Hopf-coordinate as 
\begin{eqnarray}
\label{ch-010}
&& X_{\a} \bar X_{\b} u + r(\bar X_0 u) h_{\a\bar\b}
\nonumber\\
&=& u_{\a\bar\b} + \frac{1}{2} u_t h_{\a\bar\b} + i (h_\a u_{\theta\bar\b} - h_{\bar\b} u_{\theta\a}) + h_\a h_{\bar\b} u_{\theta\theta}, 
\end{eqnarray}
and then it gives us a positive $(1,1)$-form $\Theta$ as 
$$ \Theta: = \sum_{\a, \b =1}^n \left( X_{\a} \bar X_{\b} u + r(\bar X_0 u) h_{\a\bar\b} \right) id\z^{\a} \wedge d\bar\z^{\b} \geq 0.$$
Moreover, we can also write down the mixed terms as 
\begin{equation}
\label{ch-011}
 X_{\a} \bar X_{0} u
= \frac{1}{2} \left\{ u_{r\a} - 2h_{\a} r^{-1} u_{\theta\theta}   + i ( 2r^{-1} u_{\theta\a} + h_{\a} u_{r\theta})  \right\},
\end{equation}
and 
\begin{equation}
\label{ch-012}
 X_{0} \bar X_{\a} u
= \frac{1}{2} \left\{ u_{r\bar\a} - 2h_{\bar\a} r^{-1} u_{\theta\theta}   - i ( 2r^{-1} u_{\theta\bar\a} + h_{\bar\a} u_{r\theta})  \right\}.
\end{equation}

Finally, we note that the matrix $S(u)$ boils down to the following, 
if $u$ is circular symmetric: 
\\
$$
S(u): = 
\begin{pmatrix}
 r^{-2}u_{tt}&  \frac{1}{2r} u_{t\bar\b} \\
\\
\frac{1}{2r}u_{t\a} &  u_{\a\bar\b} + \frac{1}{2}u_t h_{\a\bar\b}
\end{pmatrix}.
$$
\smallskip

\section{Appendix B: The decomposition formula}
\smallskip

In this section, 
we are going to compute the following $(2n+1)$-form on the hypersphere $S_r$ in $\bC^{n+1}$, 
and then obtain the decomposition formula for the complex Monge-Amp\`ere mass of a function $u\in \cP^{2}(B_1)$: 
$$ d^cu|_{S_r} \wedge (dd^c u |_{S_r})^n. $$
First we rewrite equation \eqref{ch-003} as 
\begin{equation}
\label{df-001}
d^c u|_{S_r} = (r u_r)\eta + \frac{i}{2} \left( \dbar_{\cD} u - \d_{\cD}u \right), 
\end{equation}
where $\d_{\cD}$ is the exterior derivative towards the directions in $\cD^{1,0}$,
 and $\dbar_{\cD}$ is its complex conjugate. 
Thanks to the commutation relations in equation \eqref{bf-0130},
we still have 
$$ \d_{\cD}^2 =0, \ \ \ \mbox{and} \ \ \ \dbar_{\cD}^2 =0.$$
However, we note that the following operator 
$$\d_{\cD}\dbar_{\cD} + \dbar_{\cD} \d_{\cD}$$ 
is no longer zero, since the commutator in equation \eqref{bf-0132} is non-trivial.  
Moreover, we also recall that the vector fields $\{\xi_0, X_{\a}, \bar{X}_{\a}\}$
form a local frame of the tangent bundle of the hypersphere, 
and the dual $\{ \eta_0, \lambda^{\a}, \bar\lambda^{\a} \}$ is a coframe field on the cotangent bundle. 
Therefore, we can decompose the exterior derivative $d_{S_r}$ on the hypersphere into the following parts: 
\begin{equation}
\label{df-0011}
d_{S_r}: = d_{\xi_0} + \d_{\cD} + \dbar_{\cD}, 
\end{equation}
where $d_{\xi_0}$ denotes by the exterior derivative along the $S^1$-direction.

\subsection{Computations}
First,  
we rewrite equation \eqref{ch-007} as follows
\begin{equation}
\label{df-002}
i\ddbar u|_{S_r} = \Xi  + \Theta, 
\end{equation}
where the $2$-form $\Xi$ is defined as 
\begin{eqnarray}
\label{df-003}
 \Xi: &=& 
 \sum_{\a =1}^n i \left\{ (X_\a \bar X_0 u)   d\z^{\a}\wedge \bar\lambda^0 + (X_0 \bar X_\a u)   \lambda^0\wedge d\bar\z^{\a} \right\}|_{S_r}  
  \nonumber\\
  &=&  \sum_{\a =1}^n \left\{ (X_\a \bar X_0 u) d\z^{\a}  + (X_0 \bar X_\a u) d\bar\z^{\a}  \right\} \wedge (2r\eta)
  \nonumber\\
  &=&  \left\{  \d_{\cD} (\bar X_0 u) +  \dbar_{\cD} (X_0 u)   \right\} \wedge (2r \eta). 
 \end{eqnarray}
Hence it follows 
\begin{equation}
\label{df-004}
( i\ddbar u|_{S_r} )^n = \Theta^n + n \Theta^{n-1}\wedge \Xi,  
\end{equation}
since we have $\eta\wedge \eta =0$. 
Combing with equation \eqref{df-001} and \eqref{df-004},  
we obtain 
\begin{equation}
\label{df-0045}
 d^cu|_{S_r} \wedge ( i\ddbar u|_{S_r} )^n 
 = (ru_r)\eta \wedge \Theta^n + n d^c_{\cD}u \wedge \Xi \wedge \Theta^{n-1},
\end{equation}
where we simply write 
$$ d^c_{\cD}u:= \frac{i }{2}(\dbar_{\cD} u - \d_{\cD} u ). $$

Next we rewrite the $(1,1)$-form $\Theta$ as 
\begin{eqnarray}
\label{df-005}
\Theta &=& i \d_{\cD} \dbar_{\cD} u  + r (\bar X_0 u) \omega
\nonumber\\
&=& - i  \dbar_{\cD} \d_{\cD} u  + r (X_0 u) \omega,
\end{eqnarray}
where $\omega: = dd^c h$ is a K\"ahler form on the base manifold $\bC\bP^n$ with $\omega = 2\omega_{FS}$, 
see equation \eqref{bf-007}. 
Then we can further write the first term on the R.H.S. of equation \eqref{df-0045} as 
\begin{eqnarray}
\label{df-006}
(ru_r) \eta \wedge \Theta^n &=& \frac{1}{4} d\theta\wedge \left( i \d_{\cD} \dbar_{\cD} u  + r (\bar X_0 u) \omega \right)^n (r \bar X_0 u)
\nonumber\\
&+ & \frac{1}{4} d\theta\wedge \left( - i \dbar_{\cD} \d_{\cD} u  + r ( X_0 u) \omega \right)^n (r X_0 u),
\end{eqnarray}
since $\Theta^n$ is an $(n,n)$-form in the base direction. 
Moreover, we have the following expansions: 
\begin{eqnarray}
\label{df-007}
\Theta^n &=& 
\sum_{k=0}^n   \left( \begin{array}{c} n \\ k \end{array} \right) ( i \d_{\cD} \dbar_{\cD} u )^k \wedge \left( r (\bar X_0 u)  \right)^{n-k} \omega^{n-k}
\nonumber\\
&=&
\sum_{k=0}^n   \left( \begin{array}{c} n \\ k \end{array} \right) ( - i \dbar_{\cD} \d_{\cD} u )^k \wedge \left( r (X_0 u)  \right)^{n-k}\omega^{n-k}.
\end{eqnarray}
On the other hand, the second term on the R.H.S. of equation \eqref{df-0045} can be computed as follows:  
\begin{eqnarray}
\label{df-008}
&& (2 n r) d^c_{\cD} u \wedge \left\{ \d_{\cD} (\bar X_0 u) + \dbar_{\cD} (X_0 u) \right\} \wedge \eta \wedge \Theta^{n-1}
\nonumber\\
&=& \frac{inr}{4} \dbar_{\cD} u \wedge \d_{\cD} (\bar X_0 u) \wedge \Theta^{n-1} \wedge d\theta
\nonumber\\
&-& \frac{inr}{4} \d_{\cD} u \wedge \dbar_{\cD} (X_0 u) \wedge \Theta^{n-1} \wedge d\theta, 
\end{eqnarray}
where we used the fact that the top degree in the base directions is $(n,n)$. 
Let $k=1,\cdots, n$, and then we are going to consider the following integrals on the sphere: 
\begin{equation}
\label{df-009}
  \int_{S_r} \d_{\cD} (\bar X_0 u) \wedge 
\dbar_{\cD} u  \wedge (i \d_{\cD} \dbar_{\cD} u)^{k-1} \wedge \left( r \bar X_0 u \right)^{n-k} \omega^{n-k} \wedge d\theta
\end{equation}
Next we utilize equation \eqref{df-0011} to obtain the following:
\begin{eqnarray}
\label{df-010}
&&d_{S_r} \left\{ (\bar X_0 u)^{n-k+1} \dbar_{\cD}u \wedge (\d_{\cD}\dbar_{\cD} u)^{k-1} \wedge \omega^{n-k} \wedge d\theta \right\}
\nonumber\\
&=& d_{\xi_0} \left\{ (\bar X_0 u)^{n-k+1} \dbar_{\cD}u \wedge (\d_{\cD}\dbar_{\cD} u)^{k-1} \wedge \omega^{n-k} \wedge d\theta \right\}
\nonumber\\
&+& (n-k+1) \d_{\cD} (\bar X_0 u) \wedge  
\dbar_{\cD} u  \wedge ( \d_{\cD} \dbar_{\cD} u)^{k-1} \wedge \left(  \bar X_0 u \right)^{n-k} \omega^{n-k} \wedge d\theta
\nonumber\\
&+& (\bar X_0 u)^{n-k +1}  ( \d_{\cD} \dbar_{\cD} u)^{k} \wedge \omega^{n-k}\wedge d\theta.
\nonumber\\
\end{eqnarray}
Here the operator $\dbar_{\cD}$ does not contribute in the above computation 
because of the degree reason again. 
Moreover, the first term on the R.H.S. of equation \eqref{df-010} vanishes 
after taking the integral along the $S^1$-direction, 
and this will be proved in Lemma \ref{lem-df-001}. 

Hence we obtain the following integration by parts for the first term
on the R.H.S. of equation \eqref{df-008}: 
\begin{eqnarray}
\label{df-011}
&& -(inr) \int_{S_r}  \d_{\cD} (\bar X_0 u) \wedge  \dbar_{\cD} u  \wedge \Theta^{n-1} \wedge d\theta
\nonumber\\
&=&  \sum_{k=1}^n  \left( \begin{array}{c} n-1 \\ k-1 \end{array} \right) \frac{n}{n-k+1}
\int_{S_r} (r\bar X_0 u)^{n-k+1} ( i\d_{\cD} \dbar_{\cD} u)^{k} \wedge \omega^{n-k}\wedge d\theta,
\nonumber\\
\end{eqnarray}
and the second term is equal to its complex conjugate as 
\begin{eqnarray}
\label{df-012}
&& (inr) \int_{S_r}  \dbar_{\cD} ( X_0 u) \wedge  \d_{\cD} u  \wedge \Theta^{n-1} \wedge d\theta
\nonumber\\
&=&  \sum_{k=1}^n  \left( \begin{array}{c} n-1 \\ k-1 \end{array} \right) \frac{n}{n-k+1}
\int_{S_r} (rX_0 u)^{n-k+1} ( -i\dbar_{\cD} \d_{\cD} u)^{k} \wedge \omega^{n-k}\wedge d\theta.
\nonumber\\
\end{eqnarray}
Then we are ready to prove the following decomposition formula. 

\begin{theorem}
\label{thm-df-001}
For any function $u\in \cP^{2}(B_1)$, we have 
\begin{eqnarray}
\label{df-013}
&& \int_{S_r} d^cu \wedge (dd^c u)^n 
 \nonumber\\
 &=& \sum_{k=0}^{n}  \left( \begin{array}{c} n+1 \\ k+1 \end{array} \right) (-1)^k 
 \int_{S_r} \left\{   (r \bar X_0 u)^{k+1} + (r X_0 u)^{k+1} \right\} \omega^k \wedge \Theta^{n-k}\wedge \frac{d\theta}{4}. 
 \nonumber\\
\end{eqnarray}
\end{theorem}
\begin{proof}
Combing with equation \eqref{df-006}, \eqref{df-007}, \eqref{df-011} and \eqref{df-012},
we obtain
\begin{eqnarray}
\label{df-014}
&& \int_{S_r} d^cu \wedge (dd^c u)^n 
 \nonumber\\
 &=& \sum_{k=0}^{n}  \left( \begin{array}{c} n+1 \\ k \end{array} \right)
 \int_{S_r}  (r \bar X_0 u)^{n-k+1} \omega^{n-k} \wedge (i \d_{\cD}\dbar_{\cD} u)^k \wedge \frac{d\theta}{4}
 \nonumber\\
 &+& \sum_{k=0}^{n}  \left( \begin{array}{c} n+1 \\ k \end{array} \right)
 \int_{S_r}  (r X_0 u)^{n-k+1} \omega^{n-k} \wedge (- i \dbar_{\cD} \d_{\cD} u)^k \wedge \frac{d\theta}{4}. 
 \nonumber\\
\end{eqnarray}

Then our decomposition formula follows, 
after applying a combinatoric identity as in Corollary (5.3), \cite{HLX}.

\end{proof}

Finally, we provide the following Lemma \ref{lem-df-001} to justify the 
integration by parts in equation \eqref{df-011} and \eqref{df-012}.  

\begin{lemma}
\label{lem-df-001}
Suppose $f$ is a smooth $2n$-form on the sphere $S_r$, 
and it has degree $(n-1, n)$ in the directions of the base manifold $\bC\bP^n$. 
Then we have 
\begin{equation}
\label{df-015}
\int_{S_r} d_{\xi_0} f =0. 
\end{equation}
\end{lemma}
\begin{proof}
Locally we can write the form $f$ as 
$$ f: = f_j d\z^1 \wedge \cdots \wedge d \hat{\z}^j \wedge \cdots \wedge d\z^n \wedge d\bar\z \wedge d\theta,$$
where $d\bar \z$ stands for the $(0,n)$ form $d\bar\z^1 \wedge \cdots \wedge d\bar \z^n$. 
Thanks to equation \eqref{bf-005}, 
we can further write the following exterior derivatives as 
\begin{eqnarray}
\label{df-016}
d_{\xi_0} f &=& \xi_0 (f_j) \eta_0 \wedge d\z^1 \wedge \cdots \wedge d \hat{\z}^j \wedge \cdots \wedge d\z^n \wedge d\bar\z \wedge d\theta
\nonumber\\
&=&  \sum_j \left\{ H_j ( \d_{\theta} f_j ) \right\}  d\theta\wedge d\z \wedge d\bar\z,
\end{eqnarray}
where $H_j: = H_j (\z, \bar\z)$ is a basic function for all $j=1,\cdots, n$.
Hence we can take the integral as 
\begin{eqnarray}
\label{df-017}
\int_{S_r} d_{\xi_0} f &=&   \sum_j \int_{\bC\bP^n} H_j d\z \wedge d\bar\z \int_{S^1} (\d_{\theta} f_j) d\theta
\nonumber\\
&=& 0. 
\end{eqnarray}
This is because $f$ is a smooth form on the sphere $S_r$, 
and then its coefficients $f_j$'s must be periodic in the $S^1$-direction. 
Therefore, our result follows.

\end{proof}


\bigskip













\bigskip
\bigskip




\bigskip



\begin{bibdiv}
\begin{biblist}


\bib{ACH19}{article}{
   author={ \AA hag, P.},
   author={Cegrell, U. },
   author={Hiep, P.-H.},
   title={On the Guedj-Rashkovskii conjecture},
   journal={Ann. Polon. Math.},
   volume={123},
   date={2019},
   number={},
   pages={15-20},
}


\bib{ACH15}{article}{
   author={ \AA hag, P.},
   author={Cegrell, U. },
   author={Hiep, P.-H.},
   title={Monge-Amp\` ere measures on subvarieties }, 
   journal={J. Math. Anal. Appl. },
   volume={423},
   date={2015},
   number={},
   pages={94-105},
}


\bib{BT0}{article}{
   author={Bedford, E.}
   author={Talyor, A.},
   title={The Dirichlet Problem for a Complex Monge-Amp\`ere equation},
   journal={Inventiones math.},
   volume={37},
   date={1976},
   number={},
   pages={1-44},
}



\bib{BT}{article}{
   author={Bedford, E.}
   author={Talyor, A.},
   title={A new capacity for plurisubharmonic functions},
   journal={Acta Math.},
   volume={149},
   date={1982},
   number={},
   pages={1-41},
}





\bib{BB}{article}{
  author={Berman, R.}
   author={Berndtsson, B.},
   title={Plurisubharmonic functions with symmetry},
   journal={ Indiana Univ. Math. J.},
   volume={63},
   date={2014},
   number={},
   pages={345-365},
}






\bib{SW24}{article}{
   author={Biard, S.},
   author={Wu, J. }
   title={ Equivalence between VMO functions and zero Lelong numbers functions},
   journal={arXiv:2403.03568},
   volume={}
   date={}
   page={}
}







\bib{Blo04}{article}{
  author={Blocki, Z.},
   title={On the definition of the Monge-Amp\`ere operator in $\bC^2$},
   journal={Math. Ann.},
   volume={328},
   date={2004},
   number={},
   pages={415-423},
  }




\bib{BFJ07}{article}{
   author={Boucksom, S.},
   author={Favre, C. },
   author={Jonsson, M.},
   title={Valuations and plurisubharmonic singularities.},
   journal={Publ. Res. Inst. Math. Sci.},
   volume={44},
   date={2008},
   number={2},
   pages={449-494},
}




\bib{BG08}{article}{
   author={Boyer, C.P.},
   author={Galicki, K.},
   title={Sasakian Geometry},
   journal={Oxford Math. Monogr.},
   volume={},
   date={2008},
   number={},
   pages={},
}













\bib{CW21}{article}{
   author={Chen, Bo-Yong},
   author={Wang, Xu} 
   title={Bergman kernel and oscillation theory of plurisubharmonic functions},
   journal={Math. Z.},
   volume={297},
   date={2021},
   number={},
   pages={1507-1527},
}









\bib{CG09}{article}{
   author={Coman, D.},
   author={Guedj, V.},
   title={Quasiplurisubharmonic Green functions.},
   journal={J. Math. Pures Appl. },
   volume={92},
   date={2009},
   number={},
   pages={456-475},
}




\bib{Dem93}{article}{
   author={Demailly, J.P.}, 
   title={Monge-Amp\`ere operators, Lelong numbers and intersection theory},
   journal={Complex analysis and geometry, Univ. Ser. Math., Plenum, New York,},
   volume={},
   date={1993},
   number={},
   pages={115-193},
}


\bib{DLLWZ}{article}{
   author={Deng, F.}
   author={Li, Y.},
   author={Liu, Q.},
   author={Wang, Z.},
   author={Zhou, X.},
   title={Log truncated threshold and zero mass conjecture},
   journal={arXiv: 2501.16669  },
   volume={},
   date={},
   page={}
}



\bib{DGZ15}{article}{
   author={Dinew, S.}
   author={Guedj, V.},
   author={Zeriahi, A.},
   title={Open problems in pluripotential theory},
   journal={arXiv: 1511.00705},
   volume={},
   date={},
   page={}
}
















\bib{GT}{article}{
   author={Gilbarg, D.}
   author={Trudinger, N.}
   title={Elliptic partial differential equations of second order},
   journal={Springer, Classics in Mathematics},
   volume={224},
   date={},
   number={},
   pages={},
}






\bib{GKN00}{article}{
   author={Godli\'nski, M.}
   author={Kopczy\'nski, W.}
   author={Nurowski, P.}
   title={Locally Sasakian manifolds},
   journal={},
   volume={},
   date={2000},
   number={},
   pages={},
}


\bib{GZ12}{article}{
   author={Guan, Pengfei}
   author={Zhang, Xi}
   title={Regularity of the geodesic equation in the space of Sasakian metrics},
   journal={Advances in Mathematics},
   volume={230},
   date={2012},
   number={},
   pages={321-371},
}





\bib{G10}{article}{
   author={Guedj, V.},
   title={Propri\'et\'es ergodiques des applications rationnelles.},
   journal={Quelques aspects des syst\`emes dynamiques polynomiaux S. Cantat, A. Chambert-Loir, V.Guedj Panoramas et Synth. },
   volume={30},
   date={2010},
   page={}
}






\bib{GZ}{article}{
   author={Guedj, V.},
   author={Zeriahi, A.},
   title={Degenerate complex Monge-Amp\`ere equations},
   journal={EMS},
   volume={},
   date={2017},
   page={}
}



\bib{H13}{article}{
   author={He Weiyong},
   title={The Sasaki-Ricci Flow and compact Sasaki manifolds of positive transversal holomorphic bisectional curvature},
   journal={J. Geom. Anal.},
   volume={23},
   date={2013},
   page={1876-1931}
}







\bib{HS16}{article}{
   author={W. He},
   author={S. Sun},
   title={Frankel conjecture and Sasaki geometry},
   journal={Advances in Mathematics},
   volume={291},
   date={2016},
   page={912-960}
}



\bib{HL21}{article}{
   author={W. He },
   author={J. Li},
   title={Geometric pluripotential theory on Sasaki manifolds},
   journal={The Journal of Geometric Analysis},
   volume={31},
   date={2021},
   page={1093-1179}
}



\bib{HLX}{article}{
   author={W. He},
   author={L. Li},
   author={X. Xu}
   title={On the residual Monge-Amp\`{e}re mass of plurisubharmonic functions with symmetry, II},
   journal={arxiv.org/2309.13288},
   volume={},
   date={},
   page={}
}














\bib{Li19}{article}{
   author={Long Li},
   title={The Lelong number, the Monge-Amp\`ere mass and the Schwarz symmetrization of plurisubharmonic functions.},
   journal={Ark. Mat.},
   volume={58},
   date={2020},
   number={},
   pages={369-392},
}







\bib{Li23}{article}{
   author={Long Li},
   title={On the residual Monge-Amp\`{e}re mass of plurisubharmonic functions with symmetry in $\bC^2$},
   journal={Math. Z.},
   volume={},
   date={},
   number={},
   pages={},
}




\bib{LW98}{article}{
   author={Fanghua Lin},
   author={Changyou Wang}
   title={Energy identity of harmonic map flows from surfaces},
   journal={Calc. Var.},
   volume={6},
   date={1998},
   number={},
   pages={369-380},
}















\bib{Ra01}{article}{
   author={Rashkovskii, A. }, 
   title={Lelong numbers with respect to regular plurisubharmonic functions},
   journal={Results Math.},
   volume={39},
   date={2001},
   number={},
   pages={320-332},
}

\bib{Ra06}{article}{
   author={Rashkovskii, A. }, 
   title={Relative types and extremal problems for plurisubharmonic functions.},
   journal={Int. Math. Res. Not. 2006 Art. ID 76283, 26 pp},
   volume={},
   date={},
   number={},
   pages={},
}






\bib{Ra13}{article}{
   author={Rashkovskii, A. }, 
   title={Analytic approximations of plurisubharmonic singularities.},
   journal={Math. Z. },
   volume={275},
   date={2013},
   number={3-4},
   pages={1217-1238},
}

\bib{Ra16}{article}{
   author={Rashkovskii, A. }, 
   title={Some problems on plurisubharmonic singularities.},
   journal={Mat. Stud.},
   volume={45},
   date={2016},
   number={},
   pages={104-108},
}






\bib{SH62}{article}{
   author={Sasaki, S.},
   author={Hatakeyama, Y.}, 
   title={On differential manifolds with contact metric structures.},
   journal={ },
   volume={},
   date={1962},
   number={},
   pages={},
}


















\bib{Wik05}{article}{
   author={Wiklund, J. }, 
   title={Plurcomplex charge at weak singularities. },
   journal={arXiv:math/0510671. },
   volume={},
   date={},
   number={},
   pages={},
}


\end{biblist}
\end{bibdiv}
\end{document}